\documentclass[red,11pt,a4paper]{article}

\usepackage{amsmath,amssymb,amsthm,color}
\theoremstyle{definition}
 \newtheorem{dfn}{Definition}
 \newtheorem{remark}[dfn]{Remark}

\theoremstyle{plain}
 \newtheorem{thm}[dfn]{Theorem}
 \newtheorem{prop}[dfn]{Proposition}
 \newtheorem{lem}[dfn]{Lemma}
 \newtheorem{cor}[dfn]{Corollary}

\numberwithin{equation}{section}
\newcommand{\bn}{{\bold n}}
\newcommand{\bu}{{\bold u}}

\newcommand{\bv}{{\bold v}}
\newcommand{\bw}{{\bold w}}
\newcommand{\ba}{{\bold a}}

\newcommand{\bff}{{\bold f}}
\newcommand{\bA}{{\bold A}}

\newcommand{\bD}{{\bold D}}

\newcommand{\bF}{{\bold F}}
\newcommand{\bG}{{\bold G}}

\newcommand{\bI}{{\bold I}}

\newcommand{\bR}{{\bold R}}
\newcommand{\bV}{{\bold V}}

\newcommand{\bU}{{\bold U}}

\newcommand{\dv}{{\rm div}\,}

\newcommand{\BR}{{\Bbb R}}
\newcommand{\BC}{{\Bbb C}}

\newcommand{\BN}{{\Bbb N}}

\newcommand{\BE}{{\Bbb E}}

\newcommand{\CA}{{\mathcal A}}
\newcommand{\CB}{{\mathcal B}}

\newcommand{\CD}{{\mathcal D}}

\newcommand{\CF}{{\mathcal F}}

\newcommand{\CL}{{\mathcal L}}

\newcommand{\CR}{{\mathcal R}}
\newcommand{\CS}{{\mathcal S}}
\newcommand{\CT}{{\mathcal T}}
\newcommand{\CH}{{\mathcal H}}

\newcommand{\CU}{{\mathcal U}}
\newcommand{\CV}{{\mathcal V}}

\newcommand{\CX}{{\mathcal X}}
\newcommand{\CY}{{\mathcal Y}}

\newcommand{\bg}{{\bold g}}
\newcommand{\bh}{{\bold h}}

\newcommand{\pd}{\partial}

\newcommand{\pf}{{\bf Proof}.~}

\newcommand{\R}{\mathbb{R}}
\newcommand{\de}{\partial}

\newcommand{\eq}[1]{\begin{equation}
\begin{split}
#1
\end{split}
\end{equation}}

\newcommand{\lr}[1]{\left( #1 \right)}

\begin{document}

\title{On the maximal $L_p$-$L_q$  regularity of solutions\\
to  a general linear parabolic system}
\author{Tomasz PIASECKI
\thanks{Corresponding Author, Institute of Applied Mathematics and Mechanics, 
University of Warsaw, Banacha 2, 02-097 Warsaw, Poland. 
E-mail address: tpiasecki@mimuw.edu.pl.
Supported by the Top Global University Project and the Polish National Science Centre grant 2018/29/B/ST1/00339.},
\quad 
Yoshihiro SHIBATA
\thanks{Department of Mathematics and Research Institute of 
Science and Engineering, 
Waseda University, 
Ohkubo 3-4-1, 
Shinjuku-ku, Tokyo 169-8555, Japan. Adjunct faculty member in the Department of Mechanical Engineering and Materias Science, University of Pittsburgh.
E-mail address: yshibata@waseda.jp.
Partially supported by  
JSPS Grant-in-aid for Scientific Research (A) 17H0109
and Top Global University Project.},
\enskip 
 and \enskip Ewelina ZATORSKA
\thanks{Department of Mathematics University 
College London Gower Street London WC1E 6BT, UK.
E-mail address: e.zatorska@ucl.ac.uk.
Supported by the Top Global University Project and the Polish Government MNiSW research grant 2016-2019 "Iuventus Plus"  No.  0888/IP3/2016/74.}}

\date{}

\maketitle

\begin{center}
{\bf Abstract}
\end{center}

\vskip5mm

\noindent We show the existence of solution in the maximal $L_p-L_q$ regularity 
framework to a class of symmetric parabolic problems on a uniformly $C^2$ domain in $\R^n$. Our approach consist 
in showing $\CR$ - boundedness of families of solution operators to corresponding resolvent problems 
first in the whole space, then in half-space, perturbed half-space and finally, using localization arguments, 
on the domain. 
{Assuming additionally boudedness of the domain we also show exponential decay of the solution.}
In particular, our approach does not require assuming a priori the uniform Lopatinskii - Shapiro condition.  \\

\noindent{MSC Classification:} 35K40, 35K51\\
\noindent{Keywords:} {\em linear parabolic system, maximal regularity, $\CR$-boundedness}

\section{Introduction}
In this paper we consider the   following initial-boundary value problem:
\begin{equation}\label{1.1}\left\{
\begin{aligned}
\sum_{\ell=1}^n R_{k\ell}(x)\pd_t u_\ell(x,t)
-\dv\lr{\sum_{\ell=1}^n B_{k\ell}(x)\nabla u_\ell(x,t)} & = F_k(x,t)
&\quad&\text{in $\Omega\times(0, T)$}, \\
\sum_{\ell=1}^n B_{k\ell}(x)\nabla u_\ell(x,t) \cdot \bn(x) & = G_k(x,t)
&\quad&\text{on $\Gamma \times (0, T)$}, \\
u_k|_{t=0} (x)& = u_{0k}(x)
&\quad&\text{in $\Omega$},
\end{aligned}\right.
\end{equation}
where $n$ is an arbitrary large natural number, $k\in\{1, \ldots, n\}$,  $\Omega$  a uniformly $C^2$ domain in 
$\BR^N$ ($N \geq 2$), $\Gamma$ is the boundary of $\Omega$,  $\bn$ is the unit outer normal vector to $\Gamma$,
$x=(x_1, \ldots, x_N)$ is a point of
 $\Omega$, and $t\in(0,T)$ is a time variable.

The $n$-vector of unknown functions is denoted by $\bu= (u_1, \ldots, u_n)^\top$ where $(\cdot)^\top$ denotes the transposed $(\cdot)$. Similarly, 
$\bF = (F_1, \ldots, F_n)^\top$, $\bG = (G_1, \ldots, G_n)^\top$,
and $\bu_0 = (u_{01}, \ldots, u_{0n})^\top$ denote given $n$-vectors of functions prescribing the right hand side of the equations, the boundary and the initial conditions, respectively.

The $n\times n$ matrices $B=[B_{k\ell}(x)]$ and $R=[R_{k\ell}(x)]$ are given and we assume that all their components 
$B_{k\ell}(x)$ and $R_{k\ell}(x)$ are uniformly H\"older continuous functions of order $\sigma>0$ and that 
$\nabla B_{k\ell}$ and $\nabla R_{k\ell}$ are integrable with some exponent $r \in (N, \infty)$, i.e. we have
\begin{equation}\label{1.2}\begin{aligned}
&|B(x)|, |R(x)| \leq M_0
\quad \text{for any $x \in \Omega$}, \quad
\|\nabla(B, R)\|_{L_r(\Omega)} \leq M_0, \\
 &|B(x) - B(y)|\leq M_0|x-y|^\sigma,
\quad
|R(x) - R(y)|\leq M_0|x-y|^\sigma 
\quad\text{for any $x, y \in \Omega$}.
\end{aligned}
\end{equation}
for some positive constant $M_0$.

Moreover, we assume that the matrices
$B$ and $R$ are positive and symmetric, and that 
there exists  constant $m_1 > 0$ for which
\begin{equation}\label{1.3}
\langle B(x)\bv, \overline{\bv} \rangle \geq m_1|\bv|^2,
\quad 
\langle R(x)\bv, \overline{\bv} \rangle \geq m_1|\bv|^2
\end{equation}
for any complex $n$-vector $\bv$ and any $x \in \Omega$. Here and in the following, $\overline{\bv}$ denotes the 
complex conjugate of $\bv$ and $\langle\cdot, \cdot\rangle$ denotes the standard inner product in $\BR^n$.

In the rest of this paper we will rather use the following more compact matrix formulation of the system  \eqref{1.1}:
\begin{equation}\label{1.1*}
R\pd_t\bu - \dv(B\nabla\bu) = \bF\quad \text{in $\Omega\times(0, T)$}, 
\quad
B(\nabla\bu\cdot\bn)=\bG \quad \text{on $\Gamma\times(0, T)$}, 
\end{equation}
subject to the initial condition: $\bu|_{t=0} = \bu_0$ in $\Omega$, where we follow the convention:
%
$$
\nabla\bu=[\de_1 \bu,\ldots,\de_n \bu], \quad
\nabla\bu\cdot\bn = \sum_{j=1}^N n^j\pd_j\bu,
$$
and divergence of a $n\times n$ matrix $A$ is understood as a vector 
$$
\dv A = [\dv[A]_{1,\cdot},\ldots,\dv[A]_{n,\cdot}]^\top,
$$  
where $[A]_{k,\cdot}$ is the k-th row of $A$, $\bn = (n_1, \ldots, n_N)^\top$, $\nabla u_\ell = (\pd_1 u_\ell, \ldots, \pd_N u_\ell)^\top$, 
$\pd_i = \pd/\pd x_i$.

\bigskip

The issue of maximal regularity for linear parabolic problems is nowadays well investigated area. The development of the theory dates back to papers  of Lopatinskii \cite{Lop} and Shapiro \cite{Sha}  from the early fifties, where certain algebraic condition was introduced that guarantees the well posedness for a class of parabolic problems. This condition, referred to as Lopatinskii-Shapiro condition (LS), corresponds to uniform, with respect to the parameter, solvability of the family of elliptic problems on a half space. The LS condition 
has been ever since assumed in many well-posedness results for parabolic problems as it provides resolvent estimates allowing to show maximal regularity for corresponding parabolic problems.  The earliest results concerning the resolvent estimates for elliptic operators satisfying this condition have been shown by Agmon \cite{Agm}, and by Arganovich and Vishik \cite{AV}.

As far as the Cauchy problems are concerned, the maximal regularity in $L_p(X)$, where $X$ is a  Banach space with the Unconditional Martingale Difference property (UMD property) has been shown  by Da Prato and Gisvard \cite{DPG},  Dore and Venni \cite{DorV}, and Pr\"uss \cite{Pruss90} and Giga and Sohr \cite{GS}, among others. For a summary of these results we refer the reader to the monograph of Amman \cite[Theorem 4.10.7]{Amann95}.  One should also mention a different approach based on potential theory applied by Ladyzhenskaya, Solonnikov and Uraltseva in \cite{LSU} to prove the maximal regularity in $L_p((0,T),L_p(G))$ for $G$ bounded and $1<p<\infty$.  

The concept of $\CR$-sectorial operators and operator-valued Fourier multipliers, essential from the point of view of the present paper, originates from the work of Weis \cite{Weis}. In this paper a characterization of the class of operators with maximal regularity was given in terms of $\CR$-boundedness of family of associated resolvent operators. This approach has been applied for the first time to show maximal $L_p$ regularity for the Cauchy problem by Kalton and Weis in \cite{KW}.     
 Further results in this spirit have been shown by Denk, Hieber and Pr\"uss in \cite{DHP}. 
In particular, Theorem 8.2 from this work concerns the maximal $L_p$-regularity for a class of parabolic initial-boundary problems. We also recommend it as a collection of auxiliary results and for  extensive list of references on the subject.

The above overview is obviously far from complete, but it should be emphasized that all above mentioned results assume a certain version of LS condition. However, for some problems this condition could be rather difficult to check. A classical way around this obstacle consist in applying energy estimates to show the existence of weak solutions and regularizing it using a priori estimates in the maximal regularity setting, see for example \cite{MZ1}, \cite{MZ2}.
Another way to solve the problem directly, without assuming the LS condition, consists in solving the problem first on the whole space, then on a half-space, further its perturbation and finally, with a standard localization procedure, on a domain. This idea has been used, for example, in the work of Enomoto and Shibata \cite{ES1}, where the maximal $L_p-L_q$ regularity of solutions was proven first for the Stokes operator and then for the compressible Navier-Stokes equations. This has been then extended in \cite{ESB} to the case of some free boundary problem.  Our strategy relies very much on the technique developed in these two papers. Let us also mention that a similar idea  in critical regularity Besov space framework has been developed in \cite{DM1}, \cite{DM2}, \cite{DZ}.

All of above mentioned results  deal with a single equation or a system of two-three equations. The main contribution of our paper is that it provides the maximal $L_p-L_q$ regularity result for arbitrary large and more general system without the LS condition. 

Symmetric parabolic systems of type \eqref{1.1} arise in particular in mathematical description of multicomponent systems with complex diffusion.  
{Equations  \eqref{1.1} can be regarded as linearization of complex systems  that model, for example,} the motion of multicomponent mixture, transport of ions, or the evolution of densities of interacting species. Although in above described models the original problem is often  non-symmetric and only positive semidefinite, it reveals entropy structure which allows to rewrite the problem in the so-called entropic variables and to reduce the problem by one equation. The resulting system is then symmetric and it is reasonable to assume or even in certain cases it is possible to show that the system is strictly parabolic. An overview of such models together with a self contained description of entropy-based approach is presented in monograph \cite{Jungel}. 
In this context the present result has been already used in a very recent work of the authors \cite{PSZ2}, where we proved the existence and maximal regularity of solutions to the Navier-Stokes type of system of $(n+1)$- component mixture. 
{
We used the main result of this paper, Theorem \ref{thm:1.1}, to generate stability and maximal $L_p-L_q$ regularity 
result for linearization of the species subsystem. 
In particular, as we are interested in short time existence, linearizing around the initial conditions we obtain 
time-indepentent coefficients.} 
Earlier,  in \cite{PSZ1} we also considered a simplified version of this system modelling the two component compressible mixture.  In that case the linearized system was reduced to a single equation, and therefore much more straightforward to deal with. 
Up to our knowledge, the only other result for such type of systems, is due to Herberg, Meyries, Pr\"uss and Wilke \cite{HMPW}, and it is restricted to the incompressible, isothermal and isobaric multicomponent flows. Rather than eliminating one equation from the system of reaction-diffusion equations and symmetrizing it using the entropy normal form, the authors work with the whole system of $(n+1)$ equations.  Its principal part is only normally elliptic on the space $\BE=\{{\rm e}\}^\top$, where ${\rm e}$ is a $(n+1)$-vector of all entries equal to 1. However, it allows for verification of the LS condition at the linear level, which we do not require here.

\subsection{Preliminaries} 
Here we recall some definitions and auxiliary results which are used in the paper. 
\begin{dfn}\label{dfn:1.2}
 We say that $\Omega$ is a uniform $C^2$ domain,
if there exist positive constants $K$, $L_1$, and $L_2$ such 
that the following assertion holds:  For any $x_0 = (x_{01},
\ldots, x_{0N}) \in \Gamma$ there exist a coordinate number
$j$ and a $C^2$ function $h(x')$ defined on $B'_{L_1}(x_0')$ such that 
$\|h\|_{H^k_\infty(B'_{a_1}(x_0'))} \leq K$ and
\begin{align*}
\Omega\cap B_{L_2}(x_0) & = \{x \in \BR^N \mid x_j > h(x') \enskip
(x' \in B'_{L_1}(x_0')) \} \cap B_{L_2}(x_0), \\
\Gamma \cap B_{L_2}(x_0) & = \{x \in \BR^N \mid x_j = h(x') \enskip
(x' \in B_{L_1}'(x_0')) \} \cap B_{L_2}(x_0).
\end{align*}
Here, we have set 
\begin{gather*}
y' = (y_1, \ldots, y_{j-1}, y_{j+1}, \ldots, y_N) \enskip (y \in \{x, x_0\}),
 \\
B'_{L_1}(x'_0)  = \{x' \in \BR^{N-1} \mid |x' - x'_0| < L_1\}, \\
B_{L_2}(x_0) = \{x \in \BR^N \mid |x-x_0| < L_2\}.
\end{gather*}
\end{dfn}
Let us also recall the definition of the Fourier transform and its inverse 
\begin{equation} \label{ft} 
\CF[f](\xi) = \int_{\BR^N} e^{-ix\cdot\xi}
f(x)\,dx,\quad
\CF_{\xi}^{-1}[g](x) = \frac{1}{(2\pi)^N}\int_{\BR^N}
e^{i\xi\cdot x} g(\xi)\,d\xi.
\end{equation}
Analogously we introduce the partial 
Fourier transform $\CF_{x'}$ and its inverse transform $\CF_{\xi'}^{-1}$ by setting 
\eq{ \label{pft}
\CF_{x'}[f](\xi', x_N) = \int_{\BR^{N-1}} e^{-ix'\cdot\xi'}
f(x', x_N)\,dx',\\
\CF_{\xi'}^{-1}[g](x) = \frac{1}{(2\pi)^{N-1}}\int_{\BR^{N-1}}
e^{i\xi'\cdot x'} g(\xi', x_N)\,d\xi',
}
where $x' = (x_1, \ldots, x_{N-1})$ and $\xi' = (\xi_1, \ldots,\xi_{N-1})$.
Next, we recall the definition of $\CR$ boundedness of a family of operators
\begin{dfn}\label{dfn:1.1}
Let $X$ and $Y$ be two Banach spaces. A family of 
operators $\CT\subset \CL(X, Y)$ is called $\CR$ bounded
on $\CL(X, Y)$ , if there exist constants 
$C>0$ and $p \in [1, \infty)$ such that for each
$m \in \BN$, $\{f_j\}_{j=1}^m \subset X^m$, and 
$\{T_j\}_{j=1}^m \subset \CT^m$, we have
$$\|\sum_{k=1}^mr_kT_kf_k\|_{L_p((0, 1), Y)}
\leq C\|\sum_{k=1}^mr_kf_k\|_{L_p((0, 1), X)}.
$$
Here, $\CL(X, Y)$ denotes the set of all bounded linear 
functions from $X$ into $Y$ and the Rademacher functions
$r_k$, $k \in \BN$, are given by
$r_k : [0, 1] \to \{-1, 1\}$; 
$t \mapsto {\rm sign}(\sin 2^k\pi t)$.  The smallest such
$C$ is called $\CR$ bound of $\CT$ on $\CL(X, Y)$,
which is denoted by $\CR_{\CL(X, Y)}\CT$.
\end{dfn} 
Finally we recall
\begin{dfn}
For any Banach space $X$,
$H^{1/2}_p(\BR, X)$ denotes the set of all $X$ valued
Bessel potential functions, $f$,  satisfying
\begin{equation}\label{1.6*}
\|f\|_{H^{1/2}_p(\BR, X)}
= \Bigl(\int_\BR \|\CF^{-1}[(1+\tau^2)^{1/4}\CF[f](\tau)]
\|^p\,{\rm d}\tau\Bigr)^{1/p} < \infty,
\end{equation}
where $\CF$ and $\CF^{-1}$ denote the Fourier transform
and the inverse Fourier transform, respectively.
\end{dfn}
To end this subsection, we introduce some fundamental properties of
$\CR$-bounded operators and Bourgain's results concerning
 Fourier multiplier theorems with scalar multiplier.
(see, e.g., \cite[Remarks 3.2 and Proposition 3.4]{DHP} 
and  \cite{bourgain86}).
\begin{prop} \label{prop:4.1}
a) Let $X$ and $Y$ be Banach spaces,
and let $\CT$ and $\CS$ be
$\CR$-bounded families in $\CL(X,Y)$.
Then, $\CT+ \CS = \{T + S \mid
T \in \CT, \enskip S \in \CS\}$ is also an $\CR$-bounded
family in
$\CL(X,Y)$ and
$$\CR_{\CL(X,Y)}(\CT + \CS) \leq
\CR_{\CL(X,Y)}(\CT) + \CR_{\CL(X,Y)}(\CS).$$

b) Let $X$, $Y$ and $Z$ be Banach spaces,  and let
$\CT$ and $\CS$ be $\CR$-bounded families in $\CL(X, Y)$ and
$\CL(Y, Z)$, respectively.  Then, $\CS\CT = \{ST \mid
T \in \CT, \enskip S \in \CS\}$ also an $\CR$-bounded
family in $\CL(X, Z)$ and
$$\CR_{\CL(X, Z)}(\CS\CT) \leq \CR_{\CL(X,Y)}(\CT)\CR_{\CL(Y, Z)}(\CS).$$

c) Let $1 < p, \, q < \infty$ and let $D$ be a domain in $\BR^N$.
Let $m=m(\lambda)$ be a bounded function defined on a subset
$\Lambda$ in $\BC$ and let $M_m(\lambda)$ be a map defined by
$M_m(\lambda)f = m(\lambda)f$ for any $f \in L_q(D)$.  Then,
$\CR_{\CL(L_q(D))}(\{M_m(\lambda) \mid \lambda \in \Lambda\})
\leq C_{N,q,D}\|m\|_{L_\infty(\Lambda)}$.

d) Let $n=n(\tau)$ be a $C^1$-function defined on $\BR\setminus\{0\}$
that satisfies the conditions $|n(\tau)| \leq \gamma$
and $|\tau n'(\tau)| \leq \gamma$ with some constant $\gamma > 0$ for any
$\tau \in \BR\setminus\{0\}$.  Let $T_n$ be the operator-valued Fourier multiplier defined by $T_n f = \CF^{-1}(n \CF[f])$
for any $f$ with $\CF[f] \in \CD(\BR, L_q(D))$.  Then,
$T_n$ can be extended to a bounded linear operator
from $L_p(\BR, L_q(D))$ into itself.  Moreover,
denoting this extension also by $T_n$, we have
$$\|T_n\|_{\CL(L_p(\BR, L_q(D)))} \leq C_{p,q,D}\gamma.
$$
{Here,} $\CD(\BR, L_q(D))$ denotes the set of all $L_q(D)$-valued
$C^\infty$-functions on $\BR$ with compact support.
\end{prop}

We finish this section with showing 
\begin{lem} 
Let $N<q \leq r \leq \infty$. Then 
\begin{equation}\label{sob:1}
\|\nabla(fg)\|_{L_q(D)} \leq C_D\{(\|g\|_{L_\infty(D)}\|\nabla f\|_{L_q(D)}
+ \|\nabla g\|_{L_r(D)}(\alpha \|\nabla f\|_{L_q(D)}
+ C_\alpha\|f\|_{L_q(D)})\}
\end{equation}
for any $\alpha \in (0, 1)$ with some constant 
$C_\alpha$ depending on $\alpha$, where $D$ is any domain in 
$\BR^N$ with uniform $C^2$ boundary. 
\end{lem}
\begin{proof}
When $r=q$, we have
$$\|\nabla(fg)\|_{L_q(D)}\leq \|\nabla f\|_{L_q(D)}\|g\|_{L_\infty(D)}
+ \|f\|_{L_\infty(D)}\|\nabla g\|_{L_q(D)}.
$$
Since $N < q=r < \infty$, by Sobolev's imbedding theorem, we have
\begin{equation}\label{sob:2}
\|\nabla(fg))\\|_{L_q(D)}\leq C_D\{\|g\|_{L_\infty(D)}\|\nabla f\|_{L_q(D)}
+ C_{q, \tau}\|\nabla g\|_{L_r(D)}\|f\|_{W^{N/q+\tau}_q(D)}\}
\end{equation}
with some small number $\tau > 0$ for which $N/q + \tau < 1$, where
$C_{q, \tau}$ is a constant depending on $q$ and $\tau$ essentially. 
When $1 < q < r$, let $s$ be a number for which $1/q = 1/r + 1/s$, and then
by H\"older's inequality, we have
$$
{\|\nabla(fg)\|_{L_q(D)}}\leq C_D\{\|g\|_{L_\infty(D)}\|\nabla f\|_{L_q(D)}
+ \|\nabla g\|_{L_r(D)}\|f\|_{L_s(D)}\}.
$$
Since $N(1/q-1/s) = N/r < 1$, by Sobolev's imbedding theorem, we have \eqref{sob:2}.  
  
Finally, by real interpolation theory,$$\|f\|_{W^{N/q+\tau}_q(D)} \leq C\|f\|_{L_q(D)}^{1-(N/q+\tau)}
\|f\|_{H^1_q(D)}^{(N/q+\tau)},
$$
and therefore we have \eqref{sob:1}. 
\end{proof}

\subsection{Main results}
In this paper, we shall prove the maximal 
$L_p$-$L_q$ regularity theorem for Eq. \eqref{1.1}:
\begin{thm} \label{thm:1.1}
Let $1 < p, q < \infty$ and $T > 0$.  Assume that 
$2/p + 1/q \not =1$ and that 
$\Omega$ is a uniformly $C^2$ domain in $\BR^N$ $(N \geq 2)$.\\
{\bf Existence.}~ 
Let $\bu_0=(u_{01}, \ldots, u_{0n})^\top \in B^{2(1-1/p)}_{q,p}(\Omega)^n$, 
$\bF \in L_p((0, T), L_q(\Omega)^n)$ and 
$\bG \in L_p(\BR, H^1_q(\Omega)^n) 
\cap
H^{1/2}_p(\BR, L_q(\Omega)^n)$ be given functions satisfying
the compatibility conditions:
\begin{equation}\label{1.4} B(\nabla\bu_0\cdot\bn) 
 = \bG(\cdot, 0)
\quad\text{on $\Gamma$}
\end{equation}
provided $2/p + 1/q < 1$.  Then, problem \eqref{1.1}
admits a unique solution $\bu = (u_1, \ldots, u_n)^\top$
with
\begin{equation} \label{1.5}
\bu \in L_p((0, T), H^2_q(\Omega)^n)
\cap H^1_p((0, T), L_q(\Omega)^n)
\end{equation}
possessing the estimate:
\begin{equation} \label{1.6}
\begin{aligned}
&\|\bu\|_{L_p((0, T), H^2_q(\Omega))}
+ \|\pd_t\bu\|_{L_p((0, T), L_q(\Omega))} \leq 
 Ce^{\gamma T}(\|\bu_0\|_{B^{2(1-1/p)}_{q,p}(\Omega)} \\
&\quad 
+ \|\bF\|_{L_p((0, T), L_q(\Omega))}
+ \|e^{-\gamma t}\bG\|_{L_p(\BR, H^1_q(\Omega))}
+ (1+\gamma^{1/2})\|e^{-\gamma t}\bG\|_{H^{1/2}_p(\BR, L_q(\Omega))})
\end{aligned}
\end{equation}
for any $\gamma \geq \gamma_0>0$ with  some constants $C$ and $\gamma_0$,
where $C$ is independent of $\gamma$. \\
{\bf Uniqueness.} Let $\bu$ be a $n$-vector of functions satisfying the 
regularity condition \eqref{1.5} and the homogeneous equations:
\begin{equation}\label{homo:1}
R\pd_t\bu -\dv(B\nabla\bu)=0 \quad\text{in $\Omega\times(0, T)$},
\quad
B(\nabla\bu\cdot\bn)|_\Gamma = 0, \quad
\bu|_{t=0} = 0,
\end{equation}
then $\bu = 0$. 
\end{thm}

To prove Theorem \ref{thm:1.1}, our approach is to use the 
$\CR$ bounded solution operator for the corresponding 
generalized resolvent problem and Weis's operator valued
Fourier multiplier theorem \cite{Weis}. Below we state
the existence theorem of such operators.

We consider the generalized resolvent problem corresponding
to Eq. \eqref{1.1*}:
\begin{equation}\label{1.8}
\lambda R\bv - \dv(B\nabla\bv) = \bff \quad
\text{in $\Omega$}, \quad B(\nabla \bv \cdot \bn) = \bg
\quad\text{on $\Gamma$},
\end{equation}
 where $\bv=(v_1, \ldots, v_n)^\top$, 
$\bff=(f_1, \ldots, f_n)^\top$ and $\bg = (g_1, \ldots, g_n)^\top$.
  We shall prove the following theorem.
\begin{thm}\label{thm:1.2} Let $1 < q < \infty$ and $0 < 
\epsilon < \pi/2$.  Assume that $\Omega$ is a unformly
$C^2$ domain in $\BR^N$. \\
{\bf Existence.}  Let 
\eq{
X_q(\Omega) & = \{(\bff, \bg) \mid \bff = (f_1, \ldots, f_n)
\in L_q(\Omega)^n, 
\quad
\bg = (g_1, \ldots, g_n)^\top \in H^1_q(\Omega)^n\},
\\
\CX_q(\Omega) & = \{(F_1, F_2, F_3) \mid 
F_1, F_2 \in L_q(\Omega)^n, \quad
F_3 \in H^1_q(\Omega)^n\}, }
with the norms
\begin{equation}\label{norm:1}\begin{aligned}
\|(\bff, \bg)\|_{X_q(\Omega)} &= \|\bff\|_{L_q(\Omega)} 
+ \|\bg\|_{H^1_q(\Omega)}, \\
\|(F_1, F_2, F_3)\|_{\CX_q(\Omega)} &= \|(F_1, F_2)\|_{L_q(\Omega)} 
+ \|F_3\|_{H^1_q(\Omega)},
\end{aligned}\end{equation}
and
\eq{
\Sigma_\epsilon = \{\lambda \in \BC\setminus\{0\} \mid
|\arg\lambda| \leq \pi-\epsilon\},
\quad
\Sigma_{\epsilon, \lambda_0}
= \{\lambda \in \Sigma_\epsilon \mid
|\lambda| \geq \lambda_0\}.
}
Then, there exist a constant $\lambda_0 > 0$ and 
an operator family $\CS(\lambda) \in 
{\rm Hol}(\Sigma_{\epsilon, \lambda_0}, 
\CL(\CX(\Omega), H^2_q(\Omega)^n))$ (holomorphic on $\Sigma_{\epsilon, \lambda_0}$) such that
for any $(\bff, \bg) \in X_q(\Omega)$
and $\lambda \in \Sigma_{\epsilon, \lambda_0}$, 
$\bv = (v_1, \ldots, v_n)^\top=
 \CS(\lambda)H_\lambda(\bff, \bg)$ with
$H_\lambda(\bff, \bg) = (\bff, \lambda^{1/2}\bg, \bg)$
is a  solution of  Eq. \eqref{1.8}.

Moreover, we have
\begin{equation} \label{1.9}
\CR_{\CL(\CX_q(\Omega), H^{2-k}_q(\Omega)^n)}
(\{(\tau\pd_\tau)^\ell(\lambda^{k/2}\CS(\lambda)) \mid
\lambda \in \Sigma_{\epsilon, \lambda_0}\})
\leq r_b
\end{equation}
for $k=0,1,2$ and $\ell=0,1$ with some constant
$r_b$, where $\lambda = \gamma + i\tau \in \BC$.
\\
{\bf Uniqueness.} Let $\bv \in H^2_q(\Omega)^n$ satisfy 
the homogeneous equations:
$$\lambda R\bv - \dv(B\nabla\bv) = 0 \quad\text{in $\Omega$},\quad
B(\nabla\bv\cdot\bn)|_\Gamma=0,
$$
then $\bv = 0$.
\end{thm}
\begin{remark}
The constant $\gamma_0$ from Theorem \ref{thm:1.1} can be chosen the same as the constant $\lambda_0$ from Theorem \ref{thm:1.2}.
\end{remark}
{
The second main result of our paper extends Theorem \ref{thm:1.1} giving a time-independent estimate 
provided boundary of the domain and zero mean assumptions on the data. 
\begin{thm}  \label{thm:1.1global}
Let $1 < p, q < \infty$ and $T = \infty$ in Theorem  \ref{thm:1.1}.
  Assume that $2/p + 1/q \not=1$ and that 
$\Omega$ is a bounded domain, whose boundary, $\Gamma$, is a compact $C^2$ 
hypersurface. Then, there exists a $\gamma_0 > 0$ for which the following
assertion holds: Let $\bu_0$, $\bF$ and $\bG$ be functions given in Theorem  \ref{thm:1.1}.
Moreover, we assume that 
\begin{gather}
\int_\Omega \bF(x, t)\,dx + \int_\Gamma \bG(x, t)\,d\sigma = 0 
\quad\text{for any $t> 0$  \quad and}\quad
\int_\Omega R\bu_0\,dx = 0, \label{cond:1} \\
\|e^{\gamma t}\bF\|_{L_p((0, \infty), L_q(\Omega))}
+ \|e^{\gamma t}\bG\|_{L_p(\BR, H^1_q(\Omega)^n)} + 
(1 + \gamma^{1/2})\|e^{\gamma t}\bG\|_{H^{1/2}_p(\BR, L_q(\Omega))}
< \infty \label{cond:2}
\end{gather}
for any $\gamma \leq \gamma_0$,
where $d\sigma$ is the surface element of $\Gamma$.   Then, 
the solution $\bu$ obtained in Theorem 6 decays exponentially, that is
$\bu$ satisfies the estimate:
\begin{align*}
&\|e^{\gamma t}\bu\|_{L_p((0, \infty), H^2_q(\Omega))} 
+ \|e^{\gamma t}\pd_t\bu\|_{L_p((0, \infty), L_q(\Omega))}\\
&\quad \leq C(\|\bu_0\|_{B^{2(1-1/p)}_{q,p}(\Omega)} 
+ \|e^{\gamma t}\bF\|_{L_p((0, \infty), L_q(\Omega))}
+ \|e^{\gamma t}\bG\|_{L_p(\BR, H^1_q(\Omega))}
+ \|e^{\gamma t}\bG\|_{H^{1/2}_p(\BR, L_q(\Omega))})
\end{align*}
for any $\gamma \leq \gamma_0$ with some constant $C$. 
\end{thm}
}
Theorem \ref{thm:1.1} can be proved by applying 
Weis' theorem \cite{Weis} to the representation formula
of solutions to \eqref{1.1} 
{given} by Theorem \ref{thm:1.2}. 
Thus, this paper is devoted to the proof of Theorem \ref{thm:1.2}
mainly. In Section \ref{sec:2} we solve the problem in the whole space.
Section \ref{sec:3} is dedicated to problem in a halfspace. This is the most technical part of the proof 
because of complexity of the solution formula. In Section \ref{sec:4} we consider a result 
in a perturbed halfspace and finally, in Section \ref{sec:5}, we use the properties of a uniform $C^2$ domains
to prove Theorem \ref{thm:1.2}. 
The two concluding sections are then dedicated to the proofs of Theorem \ref{thm:1.1} in Section \ref{sec:6}, 
{and Theorem \ref{thm:1.1global} in Section \ref{sec:7}}.

\section{Analysis in the whole space} \label{sec:2}
\subsection{Constant coefficients case}\label{subsec:2.1}

Let $x_0$ be any point of $\Omega$ and set 
$B^0= B(x_0)$ and $R^0= 
R(x_0)$.   In this subsection,
we consider the constant coefficients system
\begin{equation}\label{2.2}
\lambda R^0\bv - B^0\Delta \bv = \bff
\quad\text{in $\BR^N$}.
\end{equation}
By assumptions \eqref{1.2} and \eqref{1.3}, 
$R^0$ and $B^0$ are symmetric matrices and satisfy the following
conditions:
\begin{equation} \label{2.3}
|R^0|, |B^0| \leq M_0,
\quad
\langle R^0\ba, \overline{\ba} \rangle \geq m_1|\ba|^2, 
\quad
\langle B^0\ba, \overline{\ba} \rangle \geq m_1|\ba|^2
\end{equation}
for any $\ba \in \BC^n$.  Applying the Fourier transform to 
Eq. \eqref{2.2} gives
\begin{equation} \label{2.4}
(R^0\lambda + B^0|\xi^2|)\CF[\bv] = \CF[\bff]
\quad\text{in $\BR^N$}.
\end{equation}
\begin{lem} \label{lem:2.1} Let $0 < \epsilon < \pi/2$.
The matrix $R^0\lambda + B^0|\xi^2|$ is invertible at least
for $(\lambda, \xi) \in \Sigma_\epsilon \times (\BR^N\setminus\{0\})$ and 
there exists a constant $m_2 > 0$ depending on 
$M_0$, $m_1$ and $\epsilon$, but independent of 
$x_0 \in \Omega$, for which 
\begin{equation}\label{2.5}
|(R^0\lambda + B^0|\xi|^2)^{-1}|
\leq m_2(|\lambda| + |\xi^2|)^{-1}
\end{equation}
for any $(\lambda, \xi) \in \Sigma_\epsilon\times (\BR^N\setminus\{0\})$.
\end{lem}
\pf Let $(\lambda, \xi) \in \Sigma_\epsilon\times(\BR^N\setminus\{0\})$. 
We take $\lambda=|\lambda|(\cos\theta+ i\sin\theta)$ and we compute
\eq{
|\lambda\langle R^0\ba, \overline{\ba} \rangle+|\xi|^2\langle B^0\ba, \overline{\ba} \rangle|^2=(\langle R^0\ba, \overline{\ba} \rangle|\lambda|\cos\theta+ |\xi|^2\langle B^0\ba, \overline{\ba} \rangle)^2+(\langle R^0\ba, \overline{\ba} \rangle|\lambda|\sin\theta)^2\\
|\langle R^0\ba, \overline{\ba} \rangle|^2|\lambda|^2+2|\lambda||\xi|^2 \langle R^0\ba, \overline{\ba} \rangle \langle B^0\ba, \overline{\ba} \rangle\cos\theta+|\xi|^4|\langle B^0\ba, \overline{\ba} \rangle|^2.
}
Because $|\theta|\leq\pi-\epsilon$ thus $\cos\theta\geq \cos(\pi-\epsilon)>-1$ and so

\eq{
&|\lambda\langle R^0\ba, \overline{\ba} \rangle+|\xi|^2\langle B^0\ba, \overline{\ba} \rangle|^2\\
&\quad\geq
|\langle R^0\ba, \overline{\ba} \rangle|^2|\lambda|^2-2|\lambda||\xi|^2 \langle R^0\ba, \overline{\ba} \rangle \langle B^0\ba, \overline{\ba} \rangle|\cos(\pi-\epsilon)|+|\xi|^4|\langle B^0\ba, \overline{\ba} \rangle|^2\\
&\quad= |\cos(\pi-\epsilon)|(|\lambda|\langle R^0\ba, \overline{\ba} \rangle-|\xi|^2\langle B^0\ba, \overline{\ba} \rangle)^2\\
&\qquad +(1-|\cos(\pi-\epsilon)|)[(|\lambda|\langle R^0\ba, \overline{\ba}\rangle)^2+(|\xi|^2\langle B^0\ba, \overline{\ba} \rangle)^2]\\
&\quad\geq (1-|\cos(\pi-\epsilon)|)m_1^2|\ba|^4(|\lambda|^2+|\xi|^4).
}
Note that $|\cos(\pi-\epsilon)|=|\cos\epsilon|$, and $$1-|\cos\epsilon|=\frac{|\sin\epsilon|^2}{1+|\cos\epsilon|}\geq \frac12 |\sin\epsilon|^2,$$
therefore
\begin{equation} \label{2.6} \begin{aligned}
|\langle (R^0\lambda + B^0|\xi|^2)\ba, \overline{\ba} \rangle | &\geq C|\sin(\epsilon)|\sqrt{|\lambda|^2+|\xi|^4}|\ba|^2.
\end{aligned}
\end{equation}
Thus, if $(R^0\lambda + B^0|\xi|^2)\ba = 0$, then 
$\ba=0$, which means that the matrix
$R^0\lambda + B^0|\xi|^2$ is injection, and so
$\det(R^0\lambda + B^0|\xi|^2)\not=0$.  Thus, 
\eq{\label{Rinv}
(R^0\lambda + B^0|\xi|^2)^{-1}=[\det(R^0\lambda + B^0|\xi|^2)]^{-1}
{\rm cof}(R^0\lambda + B^0|\xi|^2)
}
exists. We now prove \eqref{2.5}.  Let
$$\tilde\lambda = \frac{\lambda}{|\lambda| + |\xi|^2},
\quad \tilde \xi_j = \frac{\xi_j}{\sqrt{|\lambda|
+ |\xi|^2}},
$$
and then $\det(R^0\lambda + B^0|\xi|^2)
= (|\lambda| + |\xi|^2)^n
\det(R^0\tilde\lambda + B^0|\tilde\xi|^2)$.  
Since $(\tilde\lambda, \tilde\xi)$ ranges on
some compact set in $\BC\times\BR^N$
as $|\tilde\lambda|+|\tilde\xi|^2=1$ for   
$(\lambda, \xi) \in \Sigma_\epsilon
\times \BR^N\setminus \{0\}$, 
there exists $\tilde m_2$ such that 
$$|\det(R^0\tilde\lambda + B^0|\tilde\xi|^2)| 
\geq \tilde m_2.
$$
This $\tilde m_2$ depends also on $\epsilon$ and $M_0$, but is independent of 
$x_0 \in \Omega$ due to \eqref{1.3}. Thus, we have 
$$|\det(R^0\lambda + B^0|\xi|^2)|
\geq \tilde m_2(|\lambda| + |\xi|^2)^n.
$$
Since the cofactor matrix of $R^0\lambda + B^0|\xi|^2$
is bounded by some constant independent of $x_0$
times $(|\lambda| + |\xi|^2)^{n-1}$, we have
\eqref{2.5}.  This completes the proof of Lemma \ref{lem:2.1}.
\qed
\vskip0.5pc
One of the main tools in proving the existence of $\CR$ bounded solution
operators in $\BR^N$ 
is the following lemma due to
Denk and Schnaubelt \cite[Lemma~2.1]{denk-schnaubelt15}
and Enomoto and Shibata
\cite[Theorem~3.3]{ES1}.
\begin{lem} \label{lem:2.2}
Let $1 < q < \infty$ and let
$\Lambda$ be a set in $\BC$.  Let $m=m(\lambda, \xi)$ be a
function defined on $\Lambda \times (\BR^N\setminus\{0\})$
which is infinitely  differentiable with respect to
$\xi \in\BR^N\setminus\{0\}$ for each $\lambda \in \Lambda$.
Assume that
for any multi-index $\alpha \in \BN^N_0$
there exists a constant $C_\alpha$ depending on
$\alpha$ and $\Lambda$ such that
\begin{equation}\label{3.6} |\pd_\xi^\alpha m(\lambda, \xi)|
\leq C_\alpha |\xi|^{-|\alpha|}
\end{equation}
for any $(\lambda, \xi) \in \Lambda\times(\BR^N\setminus\{0\})$.
Let $K_\lambda$ be an operator defined by $K_\lambda f
= \CF^{-1}_\xi[m(\lambda, \xi)\CF f(\xi)]$.  Then, the family
of operators
$\{K_\lambda \mid \lambda \in \Lambda\}$ is $\CR$-bounded on
$\CL(L_q(\BR^N))$ and
\begin{equation}\label{3.6*}
\CR_{\CL(L_q(\BR^N))}(\{K_\lambda \mid \lambda \in \Lambda\})
\leq C_{q, N}\max_{|\alpha| \leq N+1} C_\alpha
\end{equation}
with some constant $C_{q, N}$ depending only on $q$ and $N$.
\end{lem}
By Lemma \ref{lem:2.1}, we can define a solution $\bv$
of Eq. \eqref{2.2} by
\begin{equation} \label{2.7}
\bv = \CF^{-1}[(R^0\lambda + B^0|\xi|^2)^{-1}\CF[\bff](\xi)],
\end{equation}
and so for any multi-index $\alpha \in \BN_0^N$ we have
\eq{\label{2.7a}
\partial^\alpha_\xi\bv= \CF^{-1}[(i\xi)^\alpha(R^0\lambda + B^0|\xi|^2)^{-1}\CF[\bff](\xi)].}
Differentiating $(R^0\lambda + B^0|\xi|^2)^{-1}$ expressed by the formula \eqref{Rinv} w.r.t. $\xi=(\xi^1,\ldots,\xi^N)$, and $\tau$, respectively and using \eqref{2.5} we can estimate
\begin{equation}\label{2.8}\begin{aligned}
|\pd_\xi^\alpha(R^0\lambda + B^0|\xi|^2)^{-1}|
& \leq C_\alpha(|\lambda| + |\xi|^2)^{-1}|\xi|^{-|\alpha|},
\\
|\pd_\xi^\alpha((\tau\pd_\tau)(R^0\lambda + B^0|\xi|^2)^{-1})|
& \leq C_\alpha(|\lambda| + |\xi|^2)^{-1}|\xi|^{-|\alpha|}
\end{aligned}\end{equation}
for any multi-index $\alpha \in \BN_0^N$,  
$\lambda=\gamma + i\tau\in \Sigma_\epsilon$ and
$\xi\in\BR^N\setminus\{0\}$. Applying Lemma \ref{lem:2.2} to the solution operator defined by \eqref{2.7} and \eqref{2.7a} for $\alpha=1,2$, 
we have the following theorem, which is the main result
of this subsection.
\begin{thm} \label{thm:2.1}
Let $1 < q < \infty$ and 
$0 < \epsilon < \pi/2$.
Then, there exists an operator family $\CT_0(\lambda)
\in {\rm Hol}\,(\Sigma_{\epsilon,\lambda_0}, 
\CL(L_q(\BR^N)^n, H^2_q(\BR^N)^n))$ such that
for any $\lambda \in \Sigma_{\epsilon,\lambda_0}$ and $\bff
\in L_q(\BR^N)^n$, $\bv = \CT_0(\lambda)\bff$
is a unique solution of Eq. \eqref{2.2}. 

Moreover, for any $\lambda_0 > 0$ there exists 
a constant $r_b$ independent of $x_0\in \Omega$ for
which
\begin{equation} \label{2.9}
\CR_{\CL(L_q(\BR^N)^n, H^{2-k}_q(\BR^n))}
(\{(\tau\pd_\tau)^\ell(\lambda^{k/2}\CT_0(\lambda)) 
\mid \lambda \in \Sigma_{\epsilon, \lambda_0}\})
\leq r_b
\end{equation}
for $k=0,1,2$ and $\ell=0,1$.
\end{thm}

\subsection{Perturbed problem in $\BR^N$} \label{subsec:2.2}
In this subsection, we consider the case where the coefficients of the 
matrices 
$R$ and $B$ depend on $x$ variable. 
Let us fix $x_0 \in \Omega$. 
Let $M_1$ be a small positive number to be determined later. Let 
$d_0 > 0$ be a positive number such that 
\begin{equation} \label{2.2.1}
|R(x)-R(x_0)| \leq M_1,
\quad |B(x) - B(x_0)| \leq M_1
\end{equation}
for $x \in B_{d_0}(x_0)$, 
Let $\varphi$ be a function
in $C^\infty_0(\BR^N)$ which equals one for $x \in B_{d_0/2}(x_0)$
and zero for $x \not\in B_{2d_0/3}(x_0)$.  Let 
\begin{align*}
\tilde R(x) &= \varphi(x) R(x) 
+ (1-\varphi(x))R(x_0),\\
\tilde B(x) &= \varphi(x) B(x) + (1-\varphi(x))B(x_0),
\end{align*}
where $B(x)$ and $R(x)$ denote the functions extended to the whole space, 
we consider a perturbed problem:
\begin{equation} \label{p.2.1}
\lambda \tilde R \bv
-\dv(\tilde B\nabla \bv) = \bff
\quad\text{in $\BR^N$}.
\end{equation}
  In this subsection, 
we shall prove the following theorem.
\begin{thm} \label{thm:p.2} 
Assume that the coefficient
matrices $R$ and $B$ satisfy the conditions in \eqref{1.2}
with some exponent $r \in (N, \infty)$.   
Let $1 < q \leq r$ and $0 < \epsilon < \pi/2$.
Then, there exist $M_1 > 0$, $\lambda_0 > 0$ and 
an operator family $\CT_1(\lambda) \in {\rm Hol}\,
(\Sigma_{\epsilon, \lambda_0}, L(L_q(\BR^N)^n, H^2_q(\BR^N)^n))$
such that for any $\lambda \in \Sigma_{\epsilon, \lambda_0}$
and $\bff \in L_q(\BR^N)^n$, $\bv = \CT_1(\lambda)\bff$ is a
unique solution of Eq. \eqref{p.2.1} and 
$$\CR_{\CL(L_q(\BR^N)^n, H^{2-j}_q(\BR^N)^n)}
(\{(\tau\pd_\tau)^\ell(\lambda^{j/2}\CT_1(\lambda)) \mid
\lambda \in \Sigma_{\epsilon, \lambda_0}\} \leq 2r_b
$$
for $\ell=0,1$ and $j=0,1,2$ with some constant $r_b$ independent
of $x_0 \in \Omega$. 
{Here}, $\lambda_0$ and $r_b$ are the same constants as in
Theorem \ref{thm:2.1}.
\end{thm}
\begin{proof} To construct an $\CR$-bounded solution operator for
Eq. \eqref{p.2.1}, we consider the equation:
\begin{equation} \label{eq:t1a}
\lambda R(x_0)\bv
-B(x_0)\Delta \bv + \bR\bv = \bff 
\quad\text{in $\BR^N$}.
\end{equation}
{Above} we have set 
$$\bR\bv = \lambda \varphi(x)(R(x)-R(x_0))\bv 
-\dv(\varphi(x)(B(x)-B(x_0))\nabla \bv).
$$
Let $\CT_0(\lambda)$ be the $\CR$-bounded solution operator given in
Theorem \ref{thm:2.1}, and we set $\bv=\CT_0(\lambda)\bff$ in \eqref{eq:t1a}. Then, we have
{\begin{equation} \label{eq:t1b}
\lambda R(x_0)\CT_0(\lambda)\bff
-B(x_0)\Delta \CT_0(\lambda)\bff + \bR\CT_0(\lambda)\bff = (\bI + \CR(\lambda)) \bff
\quad\text{in $\BR^N$},
\end{equation}}
{where}
$$\CR(\lambda)\bff = 
\lambda\varphi(x)(R(x)-R(x_0))
\CT_0(\lambda)\bff 
-\dv(\varphi(x)(B(x)-B(x_0))
\nabla \CT_0(\lambda)\bff).
$$
Applying \eqref{sob:1} and using the 
conditions \eqref{1.2}, we have
\begin{align*}
&\|\dv(\varphi(\cdot)(B(\cdot) -B(x_0))
\nabla \CT_0(\lambda)\bff)\|_{L_q(\BR^N)}\\
&
\quad\leq CM_0(M_1 + \alpha)\|\nabla^2\CT_0(\lambda)\bff\|_{L_q(\BR^N)}
+ C_\alpha M_0\|\nabla \CT_0(\lambda)\bff\|_{L_q(\BR^N)}.
\end{align*}
By \eqref{1.2}, we also have
$$\|\lambda\varphi(\cdot)(R(\cdot) - R(x_0))
\CT_0(\lambda)\bff\|_{L_q(\BR^N)}
\leq CM_0M_1|\lambda|\|\CT_0(\lambda)\bff\|_{L_q(\BR^N)}.
$$
Using Theorem \ref{thm:2.1} and 
Proposition \ref{prop:4.1}, we have
$$\CR_{\CL(L_q(\BR^N)^n)}(\{(\tau\pd_\tau)^\ell\CR(\lambda) \mid
\lambda \in \Sigma_{\epsilon, \lambda_1}\})
\leq C\{M_0(M_1 + \alpha) + C_\alpha M_0\lambda_1^{-1/2}\}r_b.
$$
for any $\lambda_1 \geq \lambda_0$. 
Thus, choosing $M_1$ and $\alpha$ so small that 
$CM_0r_bM_1< 1/8$, $CM_0r_b\alpha < 1/8$ and choosing $\lambda_0 > 0$
so large that $CC_\alpha M_0r_b\lambda_0^{-1/2} < 1/4$,
we have 
$$\CR_{\CL(L_q(\BR^N)^n)}(\{(\tau\pd_\tau)^\ell\CR(\lambda) \mid
\lambda \in \Sigma_{\epsilon, \lambda_0}\})
\leq 1/2.
$$
Thus, we can construct the inverse operator $(\bI + \CR(\lambda))^{-1}
= \sum_{j=0}^\infty [-\CR(\lambda)]^j$. Then, taking $\tilde \bff = (\bI + \CR(\lambda))\bff$ 
in \eqref{eq:t1b} we see that  
$$
\bv = \CT_1(\lambda)\bff = \CT_0(\lambda)(\bI + \CR(\lambda))^{-1} \bff
$$ 
is a required $\CR$ bounded solution operator with $\CR$ bound: 
$$\CR_{\CL(L_q(\BR^N)^n, H^{2-j}_q(\BR^N)^n)}
(\{(\tau\pd_\tau)^\ell(\lambda^{j/2}\CT_1(\lambda)) \mid
\lambda \in \Sigma_{\epsilon, \lambda_0}\}) \leq 2r_b
$$
for $\ell=0,1$ and $j=0,1,2$.   
The uniqueness of solutions follows from the existence of 
solutions of the dual problem.  This completes the proof 
of Theorem \ref{thm:p.2}.

\end{proof}

\section{Model problem in the half-space} \label{sec:3}
Let $x_0$ be any point on $\Gamma$ and set $R^1 = R(x_0)$
and $B^1= B(x_0)$. 
In this section, we consider problem:
\begin{equation}\label{3.1}
\lambda  R^1\bv - \dv(
B^1\nabla \bv)  = \bff \quad
\text{in $\BR^N_+$}, \quad
B^1(\nabla \bv \cdot \bn_0)  = \bg
\quad\text{on $\BR^N_0$},
\end{equation}
{where}
$$\BR^N_+ = \{x=(x_1, \ldots, x_N) \mid x_N > 0\}, 
\enskip
\BR^N_0 = \{x=(x_1, \ldots, x_N) \mid x_N = 0\},
$$
and $\bn_0 = (0, \ldots, 0, -1)^\top$. 
First, we consider the case where $\bg=0$. 

\begin{thm}\label{thm:3.1}
Let $1 < q < \infty$ and 
$0 < \epsilon < \pi/2$.
Then, there exists an operator family $\CT_2(\lambda)
\in {\rm Hol}\,(\Sigma_{\epsilon}, 
\CL(L_q(\BR^N)^n, H^2_q(\BR^N)^n))$ such that
for any $\lambda \in \Sigma_\epsilon$ and $\bff
\in L_q(\BR^N_+)^n$, $\bv = \CT_2(\lambda)\bff$
is a unique solution of Eq. \eqref{3.1} with 
 $g_k=0$ $(k=1, \ldots, n)$.

Moreover, for any $\lambda_0 > 0$ there exists 
a constant $r_b$ independent of $x_0\in \Gamma$ for
which
\begin{equation} \label{3.2}
\CR_{\CL(L_q(\BR^N_+)^n, H^{2-k}_q(\BR^N_+)^n)}
(\{(\tau\pd_\tau)^\ell(\lambda^{k/2}\CT_2(\lambda)) 
\mid \lambda \in \Sigma_{\epsilon, \lambda_0}\})
\leq r_b
\end{equation}
for $k=0,1,2$ and $\ell=0,1$.
\end{thm}
\begin{proof}
Given $\bff = (f_1, \ldots, f_n)^\top$ in the right side of Eq. \eqref{3.1},
let $f^e_j$ be an even extension of $f_j$ to $x_N < 0$ defined by letting
$$f^e_j(x) = \begin{cases} f(x', x_N) \quad&\text{for $x_N > 0$}, \\
f(x', -x_N)\quad&\text{for $x_N < 0$},
\end{cases} $$
where $x' = (x_1, \ldots, x_{N-1})$. Set $\bF^e = (f_1^e, \ldots, f_n^e)^\top$
and we consider the whole space problem: 
\begin{equation}\label{3.3}
\lambda R^1 \bU  - \dv(
B^1\nabla \bU)  = \bF^e \quad
\text{in $\BR^N$}.
\end{equation}
Let 
$$\CT_2(\lambda)\bF^e(x) = \CF^{-1}[(R^1\lambda + B^1|\xi|^2)^{-1}\hat \bF^e(\xi)](x).
$$
Obviously, $\bU = \CT_2(\lambda)\bF^e$ satisfies Eq. \eqref{3.2}, 
and so in particular
$$\lambda R^1 \bU - \dv(
B^1\nabla \bU)  = \bff \quad
\text{in $\BR^N_+$}.
$$
Moreover, by Theorem \ref{thm:2.1}, $\CT_2(\lambda)$ has 
the same $\CR$-bound as in \eqref{2.9}. Thus, our task is to prove that
\begin{equation} \label{3.4}
\frac{\pd}{\pd x_N}\bU|_{x_N=0} = 0.
\end{equation}
Each term of $\CT_2(\lambda)\bF^e$ has a form: 
$$I_{kl}(x)=\frac{1}{(2\pi)^N}\int_{\BR^N}e^{ix\cdot\xi}\frac{\lambda^{n-1-\ell}
|\xi|^{2\ell}}{\det(R^1\lambda + B^1|\xi|^2)}\hat f_k^e(\xi)\,d\xi
$$
for some $k \in \{1 \ldots n\}, \; \ell \in \{1 \ldots n-1\}$. Thus, 
\begin{align*}
\pd_NI_{kl}|_{x_N=0} = \frac{1}{(2\pi)^N}\int_{\BR^N}e^{ix'\cdot\xi'}
\frac{\lambda^{n-1-\ell}|\xi|^{2\ell}i\xi_N\hat f_k^e(\xi)}
{\det(R^1\lambda + B^1|\xi|^2)}\,d\xi.
\end{align*}
Applying the Fourier transform with respect to $x'$, we have 
\begin{align*}
{\mathcal F}^{-1}_{x'}(\pd_NI_{kl}|_{x_N=0})(\xi') & = 
\frac{1}{2\pi}\int^\infty_{-\infty} {\mathcal F}^{-1}_{x'} \left[ \int_{\BR^{N-1}}e^{ix'\cdot\xi'}
\frac{\lambda^{n-1-\ell}|\xi|^{2\ell}i\xi_N\hat f_k^e(\xi)}
{\det(R^1\lambda + B^1|\xi|^2)}\,d\xi' \right] \\
& =
\frac{1}{2\pi}\int^\infty_{-\infty}
\frac{\lambda^{n-1-\ell}|\xi|^{2\ell}
i\xi_N\hat f^e(\xi)}{\det(R^1\lambda + B^2|\xi|^2)}\,d\xi_N
\\
& = \frac{1}{2\pi}\int^\infty_{-\infty}
\frac{\lambda^{n-1-\ell}|\xi|^{2\ell}i\xi_N}
{\det(R^1\lambda + B^2|\xi|^2)}
\int^\infty_0(e^{-iy_N\xi_N} + e^{iy_N\xi_N})
\hat f(\xi', y_N)\,dy_N \\
& = \lambda^{n-1-\ell}\int^\infty_0\hat f(\xi', y_N)\,dy_N
\frac{1}{2\pi}\int^\infty_{-\infty}
\frac{|\xi|^{2\ell}i\xi_N(e^{-iy_N\xi_N} + e^{iy_N\xi_N})}
{\det(R^1\lambda + B^1|\xi|^2)}\,d\xi_N.
\end{align*}
Thus, in order to show \eqref{3.4} it is enough to prove that
\begin{equation} \label{3.5}
\frac{1}{2\pi}\int^\infty_{-\infty}
\frac{|\xi|^{2\ell}i\xi_N(e^{-iy_N\xi_N} + e^{iy_N\xi_N})}
{\det(R^1 \lambda + B^1|\xi|^2)}\,d\xi_N
= 0.
\end{equation}
We can write  
\begin{equation} \label{3.5b}
\det(R^1\lambda + B^1|\xi|^2) 
= a_0|\xi|^{2n} + \sum_{j=1}^n a_j\lambda^j
|\xi|^{2(n-j)}. 
\end{equation}
Let $t=|\xi|^2$, then \eqref{3.5b} rewrites as
$$a_0t^n + \sum_{j=1}^n a_j\lambda^j
t^{n-j}
= a_0\prod_{j=1}^m(t+k_j|\lambda|)^{n_j},  
$$
where $m$ and $n_j$ are constants depending on $\lambda$ for which
$n = \sum_{j=1}^mn_j$ and $k_j$ are functions with respect to
$\lambda/|\lambda|$ such that $k_j\not=k_\ell$ for 
$j\not=\ell$.  In view of  \eqref{2.5}, 
$a_0\prod_{j=1}^m(t+k_j|\lambda|)^{n_j}\not=0$ for $t \geq 0$ and $\lambda \in \Sigma_\epsilon$, 
and so $k_j\not\in (-\infty, 0)$
for $\lambda \in \Sigma_\epsilon$. Thus, we have
\begin{equation} \label{det:form}
\det(R^1\lambda + B^1|\xi|^2) 
= a_0\prod_{j=1}^n(\xi_n^2 + |\xi'|^2 + k_j|\lambda|)^{n_j}
= a_0\prod_{j=1}^m(\xi_n + i\omega_j)^{n_j}(\xi_n-i\omega_j)^{n_j}
\end{equation}
with $\omega_j = \sqrt{|\xi'|^2 + k_j|\lambda|}$ where we take
${\rm Re}\, \omega_j > 0$. We rewrite the lhs of \eqref{3.5}:  
\begin{lem}
We have
\begin{equation} \label{3.5c}
\frac{1}{2\pi}\int^\infty_{-\infty}
\frac{|\xi|^{2\ell}i\xi_N(e^{-iy_N\xi_N} + e^{iy_N\xi_N})}
{\det(R^1 
\lambda + B^1|\xi|^2)}\,d\xi_N \\
=\sum_{j=1}^m \frac{1}{(n_j-1)!}J_j
\end{equation}
with 
\eq{\label{3.5d}
J_j &= 
\bigl(\frac{\pd}{\pd \xi_N}\bigr)^{n_j-1}\frac{f_j(\xi_N^2)i\xi_Ne^{iy_N\xi_N}}
{(\xi_N+i\omega_j)^{n_j}}\Bigl|_{\xi_N=i\omega_j}
- \bigl(\frac{\pd}{\pd \xi_N}\bigr)^{n_j-1}
\frac{f_j(\xi_N^2)i\xi_Ne^{-iy_N\xi_N}}
{(\xi_N-i\omega_j)^{n_j}}\Bigl|_{\xi_N=-i\omega_j} \\
&:=J_j^+ - J_j^-,
}
where we have set
$$
f_j(y)=\frac{[|\xi'|^2+y]^l}{\prod_{\ell\not=j}(y+|\xi'|^2
+k_\ell|\lambda|)^{n_\ell}}.
$$
\end{lem}
\begin{proof}  
The proof follows by direct computation of the integral on the l.h.s. of \eqref{3.5c} as a limit of curve integrals of a complex function which are computed using residue theorem. Denoting the 
integrand by $f(\xi_N)$ we have
\begin{equation}
\int^\infty_{-\infty}f(\xi_N)\,d\xi_N = 
{\rm lim}_{R \to \infty} \int_{\gamma^+_R}f(\xi_N)\,d\xi_N 
- {\rm lim}_{R \to \infty} \int_{L^+_R}f(\xi_N)\,d\xi_N
\end{equation}
where 
$$
L^+_R = \{ z \in \BC: {\rm Im}z>0, \; |z|=R \}, \quad 
\gamma^+_R = [-R,R] \times \{{\rm Im}z=0\} \cup L^+_R.
$$
Writing $\xi_N=a+bi$ we easily verify that the integral over $L^+_R$ vanishes as $R \to \infty$, and therefore by residue 
theorem the integral in the lhs of \eqref{3.5c} will be equal 
to sum of residua of the integrand on the upper complex halfplane. In order to compute the residua notice that by \eqref{det:form} we have 
\begin{equation*}
f(\xi_N)=\frac{|\xi|^2 i\xi_N e^{iy_N}}{\Pi_{j=1}^m(\xi_N-i\omega_j)^{n_j}},    
\end{equation*}
therefore in a neighbourhood of $\xi_N=i\omega_j$ we have
\begin{equation*}
f(\xi_N)=\frac{g_j(\xi_N)}{(\xi_N-i\omega_j)^{n_j}}, 
\end{equation*}
where 
\begin{equation*}
g_j(\xi_N) = \frac{f_j(\xi_N^2)i\xi_Ne^{iy_N\xi_N}}{(\xi_N+i\omega_j)^{n_j}}    
\end{equation*}
is holomorphic, which implies the form of $J_j^+$ in \eqref{3.5d}. The part with $e^{-iy_N\xi_N}$ is calculated in the same way extending the integral to a curve contained in lower complex hyperplane leading to the form of $J_j^-$. 
\end{proof}
It is easy to observe that 
\begin{lem} \label{L:df} We have the following identities
\begin{align} \label{eq:df}
\pd_N^{2\ell-1}f_j(\xi_N^2) 
& = a_0^{(2\ell-1)}f_j^{(2\ell-1)}(\xi_N^2)\xi_N^{2\ell-1} 
+ a^{(2\ell-1)}_1f^{(2\ell-2)}_j(\xi_N^2)\xi_N^{2\ell-3}
+ \cdots \nonumber \\ 
&\quad+ 
a^{(2\ell-1)}_{\ell-2}f^{(\ell+1)}_j(\xi_N^2)\xi_N^3 
+ a^{(2\ell-1)}_{\ell-1}f_j^{(\ell)}(\xi_N^2)\xi_N,\\
\pd_N^{2\ell}f_j(\xi_N^2) 
& = a_0^{(2\ell)}f_j^{(2\ell)}(\xi_N^2)\xi_N^{2\ell} 
+ a^{(2\ell)}_1f^{(2\ell-1)}_j(\xi_N^2)\xi_N^{2\ell-2}
+ \cdots \nonumber \\
&\quad+ 
a^{(2\ell)}_{\ell-1}f^{(\ell+1)}_j(\xi_N^2)\xi_N^2 
+ a^{(2\ell)}_{\ell}f_j^{(\ell)}(\xi_N^2), \nonumber \\
\end{align}
with some coefficients $a^{(k)}_m$, where $f^{(\ell)}_j = 
\pd^\ell f_j/ \pd y_j$.
\end{lem}
\begin{proof} It is enough to observe that 
$$
\de_{\xi^N} [f^{2l-k}(\xi_N^2)\xi_N^{2(l-k)-1}]=a_{kl}\xi_N^{2(l-k-1)}f^{(2l-k)}(\xi_N^2)+2f^{(2l-k+1)}(\xi_N^2)\xi_N^{2(l-k)} 
$$ 
and 
$$
\de_{\xi^N} [f^{2l-k}(\xi_N^2)\xi_N^{2(l-k)}]=b_{kl}\xi_N^{2(l-k)-1}f^{(2l-k)}(\xi_N^2)+2f^{(2l-k+1)}(\xi_N^2)\xi_N^{2(l-k)+1} 
$$
for some coefficients $a_{kl},b_{kl}$, therefore \eqref{eq:df} follows by induction.
\end{proof}
By Lemma \ref{L:df}, there exist
some functions $g_j^{(\ell)}$ for which 
\begin{equation} \label{3.6a} 
\pd_N^{2\ell-1}f_j(\xi_N^2)|_{\xi_N = \pm i\omega_j}
= \pm ig^{(2\ell-1)}_j(\omega_j^2)\omega_j, 
\quad
\pd_N^{2\ell}f_j(\xi_N^2)|_{\xi_N = \pm i\omega_j}
= g^{(2\ell)}_j(\omega_j^2).
\end{equation}
Notice that the uppercase index $g^{(l)}$ does not denote differentiation contrarily to $f^{(l)}$.
By 
{Leibniz} rule, we have 
\begin{align*}
\lr{\frac{\pd}{\pd \xi_N}}^{n_j-1}&
\frac{f_j(\xi_N^2)i\xi_Ne^{\pm iy_N\xi_N}}
{(\xi_N\pm i\omega_j)^{n_j}}\Bigl|_{\xi_N=\pm i\omega_j}=\\
&= \sum_{k=0}^{n_j-1}\de_{\xi_N}^k(i\xi_N)\de_{\xi_N}^{n_j-1-k}\frac{f_j(\xi_N^2)i\xi_Ne^{\pm iy_N\xi_N}}
{(\xi_N\pm i\omega_j)^{n_j}}\Bigl|_{\xi_N=\pm i\omega_j}
\\
&= i\xi_N \bigl( \frac{\de}{\de \xi_N} \bigr)^{n_j-1} 
\frac{f_j(\xi_N^2)e^{\pm iy_N\xi_N}}
{(\xi_N\pm i\omega_j)^{n_j}}\Bigl|_{\xi_N=\pm i\omega_j}+ i \bigl( \frac{\de}{\de \xi_N} \bigr)^{n_j-2}\frac{f_j(\xi_N^2)e^{\pm iy_N\xi_N}}
{(\xi_N\pm i\omega_j)^{n_j}}\Bigl|_{\xi_N=\pm i\omega_j}\\ 
&=
\sum_{k_1+k_2+k_3=n_j-1} C_{k_1,k_2,k_3}^{(n_j-1)} L_1^\pm + 
\sum_{k_1+k_2+k_3=n_j-2} C_{k_1,k_2,k_3}^{(n_j-2)}iL_2^\pm,
\end{align*}
where
\begin{align*}
L_1^\pm&=i\xi_N
\bigl(\frac{\pd}{\pd\xi_N}\bigr)^{k_1}f_j(\xi_N^2)
\bigl(\frac{\pd}{\pd\xi_N}\bigr)^{k_2}e^{\pm i\xi_Ny_N}
\bigl(\frac{\pd}{\pd\xi_N}\bigr)^{k_3}(\xi_N\pm i\omega_j)^{-n_j}
\Bigl|_{\xi_N=\pm i\omega_j},\\
L_2^\pm& =
\bigl(\frac{\pd}{\pd\xi_N}\bigr)^{k_1}f_j(\xi_N^2)
\bigl(\frac{\pd}{\pd\xi_N}\bigr)^{k_2}e^{\pm i\xi_Ny_N}
\bigl(\frac{\pd}{\pd\xi_N}\bigr)^{k_3}(\xi_N\pm i\omega_j)^{-n_j}
\Bigl|_{\xi_N=\pm i\omega_j}
\end{align*}
with some permutation numbers $C_{k_1,k_2,k_3}^{(n_j-1)}$
and $C_{k_1,k_2,k_3}^{(n_j-2)}$. 
Now our goal is to show
\begin{equation} \label{3.6b}
L_i^+=L_i^-, \quad i=1,2.
\end{equation}
Then by \eqref{3.5d} we have $J_j^+=J_j^-$, and therefore \eqref{3.5} holds due to \eqref{3.5c}.
Let us start with observing that 
\begin{equation} \label{3.6c}
\de_{\xi_N}^{k_2} \left. e^{\pm i\xi_N y_N}\right|_{\xi_N = \pm i\omega_j} 
= (\pm i y_N)^{k_2} \left. e^{\pm i\xi_N y_N}\right|_{\xi_N = \pm i\omega_j}
= (\pm i y_N)^{k_2} e^{-\omega_j y_N}
\end{equation}
and 
\begin{equation} \label{3.6d}
\de_{\xi_N}^{k_3}(\xi_N \pm i\omega_j)^{-n_j}=d_{k3}(\xi_N \pm i\omega_j)^{-n_j-k_3},
\end{equation}
where $d_{k3} = (-n_j)(-n_j-1)\cdots(-n_j-k_3+1)$. 
In order to show that $L_1^+=L_1^-$ we assume $k_1+k_2+k_3=n_j-1$ and consider 
first the case when $k_1$ is odd. Using \eqref{3.6a},\eqref{3.6c} and \eqref{3.6d} we get
\eq{\label{3.6e}
L_1^\pm & = (\mp \omega_j)(\pm ig_j^{(k_1)}(\omega_j^2)\omega_j)(\pm iy_N )^{k_2}e^{-\omega_j y_N}2^{-n_j-k_3}d_{k_3}(\pm i\omega_j)^{-n_j-k_3}  \\
& = (\pm i)^{k_2}(\pm i)^{-k_1-k_2-2k_3-1} (- i) g_j^{(k_1)}(\omega_j^2)\omega_j^2e^{-\omega_j y_N}(2\omega_j)^{-n_j-k_3}d_{k_3}y_N^{k_2}  \\
& = - i^{-2k_3-k_1} g_j^{(k_1)}(\omega_j^2)\omega_j^2e^{-\omega_j y_N}(2\omega_j)^{-n_j-k_3}d_{k_3}y_N^{k_2} 
}
where we have used that $(\pm 1)^{2k_3+k_1+1}=1$ because $k_1$ is odd. 
In the same manner, when $k_1$ is even, we have 
\eq{ \label{3.6f}
L_1^\pm & = \mp \omega_j g_j^{(k_1)}(\omega_j^2)(\pm iy_N )^{k_2}e^{-\omega_j y_N}2^{-n_j-k_3}d_{k_3}(\pm i\omega_j)^{-n_j-k_3} \\
& = (\pm i)^{k_2}(\pm i)^{-k_1-k_2-2k_3}(\mp \omega_j)(\pm \omega_j){^{-1}} g_j^{(2l)}(\omega_j^2)\omega_j y_N^{k_2}e^{-\omega_j y_N}2^{-n_j-k_3}d_{k_3}\omega_j^{-k_1-k_2-2k_3}\\
& = -i^{-k_1-2k_3}\omega_j g_j^{(2l)}(\omega_j^2)\omega_j y_N^{k_2}e^{-\omega_j y_N}2^{-n_j-k_3}d_{k_3}\omega_j^{-k_1-k_2-2k_3}, 
}
since this time $(\pm 1)^{k_1+2k_3} = 1$ because $k_1$ is even. 

In order to show that $L_2^+=L_2^-$ we assume $n_j=k_1+k_2+k_3+2$
and again consider first $k_1$ odd. Then  
\eq{ \label{3.6g}
L_2^\pm & = \pm i g_j^{(k_1)}(\omega_j^2)\omega_j (\pm iy_N)^{k_2}e^{-\omega_j y_N}d_{k_3}2^{-n_j-k_3}(\pm i \omega_j)^{-n_j-k_3}\\
& = (\pm i)^{1+k_2}(\pm i)^{-k_1-k_2-2k_3-2}g_j^{(k_1)}(\omega_j^2)\omega_j y_N^{k_2}e^{-\omega_j y_N}d_{k_3}(2\omega_j)^{-n_j-k_3}\\
& = i^{-k_1-2k_3-1}g_j^{(k_1)}(\omega_j^2)\omega_j y_N^{k_2}e^{-\omega_j y_N}d_{k_3}(2\omega_j)^{-n_j-k_3},
}
where we have used $(\pm 1)^{-k_1-2k_3-1} = 1$ because $k_1$ is odd. 
When $k_1$ is even, we have 
\eq{ \label{3.6h}
L_2^\pm & = g_j^{(k_1)}(\omega_j^2)(\pm iy_N)^{k_2}e^{-\omega_j y_N}d_{k_3}2^{-n_j-k_3}(\pm i \omega_j)^{-n_j-k_3}\\
& = (\pm i)^{k_2}(\pm i)^{-k_1-k_2-2k_3-2}g_j^{(k_1)}(\omega_j^2)y_N^{k_2}e^{-\omega_j y_N}d_{k_3}(2\omega_j)^{-n_j-k_3}\\ 
& = i^{-k_1-2k_3-2}g_j^{(2l)}(\omega_j^2)y_N^{k_2}e^{-\omega_j y_N}d_{k_3}(2\omega_j)^{-n_j-k_3},
}
since $(\pm 1)^{-k_1-2k_3-2} = 1$ as $k_1$ is even. 
From \eqref{3.6e}-\eqref{3.6h} we conclude \eqref{3.6b}, which leads to \eqref{3.5} as explained above. This completes the proof
of 
{Theorem \ref{thm:3.1}}.
\end{proof}
We now prove the existence of an $\CR$ bounded solution operator 
for Eq. \eqref{3.1}.
\begin{cor}\label{cor:2.1}Let $1 < q < \infty$ and 
$0 < \epsilon < \pi/2$. 
Let $X_q(\BR^N_+)$ and $\CX_q(\BR^N_+)$ be spaces  defined by 
replacing $\Omega$ by $\BR^N_+$ in Theorem \ref{thm:1.2}.
Then, there exists an operator family $\CT_3(\lambda)
\in {\rm Hol}\,(\Sigma_{\epsilon}, 
\CL(\CX_q(\BR^N_+), H^2_q(\BR^N)^n))$ such that
 $\bv = \CT_3(\lambda)(\bff, \lambda^{1/2}\bg, \bg)$
is a unique solution of Eq. \eqref{3.1}
for any $\lambda \in \Sigma_\epsilon$ and $(\bff, \bg)
\in X_q(\BR^N_+)^n$. 

Moreover, for any $\lambda_0 > 0$ there exists 
a constant $r_b$ independent of $x_0\in \Gamma$ for
which
\begin{equation} \label{3.7}
\CR_{\CL(\CX_q(\BR^N_+), H^{2-k}_q(\BR^N+))}
(\{(\tau\pd_\tau)^\ell(\lambda^{k/2}\CT_3(\lambda)) 
\mid \lambda \in \Sigma_{\epsilon, \lambda_0}\})
\leq r_b
\end{equation}
for $k=0,1,2$ and $\ell=0,1$.
\end{cor}
\begin{proof}
Notice that $\nabla\bv\cdot\bn_0 = -\pd_N\bv$.
Let $\bh = (B^1)^{-1}\bg$, 
{and consider} the boundary value problem:
\begin{equation} \label{3.8}
\lambda\bw - \Delta\bw = 0 \quad\text{in $\BR^N_+$},\quad
\pd_N\bw = -\bh.\ \quad\text{on $\BR^N_0$}.
\end{equation}
To define a solution operator of Eq. \eqref{3.8}, we use the partial 
Fourier transform $\CF_{x'}$ and its inverse transform $\CF_{\xi'}^{-1}$ defined in \eqref{pft}. 
Applying the partial Fourier transform to 
Eq. \eqref{3.8}, we have
$$((\lambda+|\xi'|^2) - \pd_N^2)\CF_{x'}[\bw](\xi', x_N) =0
\quad\text{on $(0, \infty)$}, \quad 
\pd_N\CF_{x'}[\bw](\xi', x_N)|_{x_N=0} = -\CF_{x'}[\bh](\xi', 0).
$$
Thus, we have
$$\CF_{x'}[\bw](\xi', x_N) = \frac{e^{-\sqrt{\lambda + |\xi'|^2}x_N}}
{\sqrt{\lambda + |\xi'|^2}}\CF_{x'}[\bh](\xi', 0).
$$
And so, we define a solution operator $U(\lambda)$ by
setting
$$U(\lambda)\bh = 
\CF_{\xi'}^{-1}\Bigl[\frac{e^{-\sqrt{\lambda+ |\xi'|^2}x_N}}{
\sqrt{\lambda + |\xi'|^2}}\CF_{x'}[\bh](\xi', 0)\Bigr](x').
$$
By the Volevich trick: 
$$f(x_N)g(0) = -\int^\infty_0\pd_Nf\big((x_N+y_N)g(y_N)\big)\,dy_n,
$$
with 
$$
f(x_N) = \frac{e^{-\sqrt{\lambda+ |\xi'|^2}x_N}}{
\sqrt{\lambda + |\xi'|^2}}, \quad g(y_N) = \CF_{x'}[\bh](\xi', y_N)
$$
we write $U(\lambda)\bh$ as 
\begin{align*}
&U(\lambda)\bh \\
&= -\int^\infty_0\CF_{\xi'}^{-1}\Bigl[
\frac{e^{-\sqrt{\lambda+|\xi'|^2}(x_N+y_N)}}{\sqrt{\lambda+|\xi'|^2}}
(\pd_N\CF_{x'}[\bh](\xi', y_N) - \sqrt{\lambda + |\xi'|^2}
\CF_{x'}[\bh](\xi', y_N))\Bigr](x')\,dy_N \\
&= -\int^\infty_0\CF_{\xi'}^{-1}\Bigl[
\frac{e^{-\sqrt{\lambda+|\xi'|^2}(x_N+y_N)}}{\sqrt{\lambda+|\xi'|^2}}
\Bigl(\CF_{x'}[\pd_N\bh](\xi', y_N) 
- \frac{\lambda^{1/2}}{\sqrt{\lambda + |\xi'|^2}}
\CF_{x'}[\lambda^{1/2}\bh](\xi', y_N) \\
&\phantom{-\int^\infty_0\CF_{\xi'}^{-1}\Bigl[
\frac{e^{-\sqrt{\lambda+|\xi'|^2}(x_N+y_N)}}{\sqrt{\lambda+|\xi'|^2}}
\Bigl(}
+ \sum_{j=1}^{N-1}\frac{i\xi_j}{\sqrt{\lambda + |\xi'|^2}}
\CF_{x'}[\pd_j\bh](\xi', y_N)\Bigr)\Bigr](x')\,dy_N.
\end{align*}
Let $\CY_q(\BR^N_+) = \{(F_2, F_3) \mid F_2 \in L_q(\BR^N_+)^N, \enskip
F_3 \in H^1_q(\BR^N_+)^N\}$.  And then, 
we define an operator $\CU(\lambda)$ acting on 
$(F_2, F_3) \in \CY_q(\BR^N_+)$ by 
letting 
\begin{align*}
&\CU(\lambda)(F_2, F_3) \\
& = -\int^\infty_0\CF_{\xi'}^{-1}\Bigl[
\frac{e^{-\sqrt{\lambda+|\xi'|^2}(x_N+y_N)}}{\sqrt{\lambda+|\xi'|^2}}
\Bigl(\CF_{x'}[\pd_NF_3](\xi', y_N) 
- \frac{\lambda^{1/2}}{\sqrt{\lambda + |\xi'|^2}}
\CF_{x'}[F_2](\xi', y_N) \\
&\phantom{-\int^\infty_0\CF_{\xi'}^{-1}\Bigl[
\frac{e^{-\sqrt{\lambda+|\xi'|^2}(x_N+y_N)}}{\sqrt{\lambda+|\xi'|^2}}
\Bigl(}
+ \sum_{j=1}^{N-1}\frac{i\xi_j}{\sqrt{\lambda + |\xi'|^2}}
\CF_{x'}[\pd_jF_3](\xi', y_N)\Bigr)\Bigr](x')\,dy_N,
\end{align*}
and then we have
$$U(\lambda)\bh = \CU(\lambda)(\lambda^{1/2}\bh, \bh).
$$
Moreover,  using the same argument as in \cite[Sect. 5]{SS1}, we see that
\begin{equation}\label{3.9}
\CR_{(\CL_q(\BR^N_+), H^{2-j}_q(\BR^N)^n)}
(\{(\tau\pd_\tau)^\ell(\lambda^{j/2}\CU(\lambda)) \mid
\lambda \in \Sigma_{\epsilon, \lambda_0}\}) \leq r_b
\end{equation}
for $\ell=0,1$ and $j=0,1,2$, where $r_b$ is a constant depending on
$\epsilon$, $\lambda_0 >0$, $M_0$ and $m_1$, but independent of $x_0
\in \Gamma$. 

Let $\CT_2(\lambda)$ be the operator given in Theorem \ref{thm:2.1}. 
Letting $\bF = \lambda R^1 U(\lambda)\bh - \dv(B^1\nabla U(\lambda)\bh)$, 
and setting $\bv = \CT_2(\lambda)(\bff- \bF) + U_0(\lambda)\bh$ with
$\bh = (B^1)^{-1}\bg$, 
we see that $\bv$ is a unique solution of 
Eq. \eqref{3.1}.  
{The uniqueness} follows from the existence of
solutions of the dual problem. Thus, combining Theorem \ref{thm:2.1}
and \eqref{3.9}, we have Corollary \ref{cor:2.1}.  This completes the proof.
\end{proof}
\section{Analysis in a bent half-space}\label{sec:4}
Let $\Phi$ be a diffeomorphism of $C^1$ class on $\BR^N$ and $\Phi^{-1}$
the inverse of $\Phi$.  We assume that 
$\nabla\Phi = \CA + \CB(x)$ and $\nabla\Phi^{-1} = \CA^{-1} + \CB_{-1}(y)$,
where $\CA$ is an orthogonal matrix with constant coefficients, 
$\CA^{-1}$ is the inverse matrix of $\CA$, and $\CB(x)$ and $\CB_{-1}(y)$ are 
matrices of $C^0(\BR^N)$ functions satisfying the conditions:
\begin{equation}\label{4.1}
\|(\CB,\CB_{-1})\|_{L_\infty(\BR^N)} \leq M_1,
\quad\|\nabla(\CB, \CB_{-1})\|_{L_r(\BR^N)} \leq C_K.
\end{equation}
{In the above formula} $r$ is an exponent such that $N < r < \infty$ and $C_K$ is a constant
depending on the constants $K$, $L_1$ and $L_2$ appearing in
Definition \ref{dfn:1.2}. We choose $M_1$ small enough eventually,
and so we may assume that $0 < M_1 \leq 1 \leq C_K$ without loss of 
generality. Let 
$$\Omega_+ = \Phi(\BR^N_+) = \{y = \Phi(x) \mid x \in \BR^N_+\},
\quad
\Gamma_+ = \Phi(\BR^N_0) = \{y = \Phi(x) \mid x \in \BR^N_0\}.
$$
Let $\bn_+$ be the unit outer normal to $\Gamma_+$ and 
let $\pd_{\bn_+} = \bn_+\cdot\nabla$.  Let 
$y_0$ be any point of $\Gamma_+$ and we fix it.  We assume in this 
section that there exist a positive number $d_0$ for which
\begin{equation}\label{4.2.1}
|R(y)-R(y_0)| \leq M_1,
\quad |B(y) - B(y_0)| \leq M_1
\end{equation}
for any $y \in B_{d_0}(y_0)$.  Moreover, let $M_2$ be a number for which
\begin{equation} \label{4.2.2}
\|\nabla (R, B)\|_{L_r(\BR^N)} \leq M_2.
\end{equation}
Note that since $R$ and $B$ are the extensions of functions defined on $\Omega$, due to \eqref{1.2}, we may take $M_2=M_2(M_0)$.
We may assume that 
\begin{equation} \label{4.2.3}
C_K \leq M_2.  
\end{equation}
Let $\varphi(y)$ be a function in 
$C^\infty(\BR^N)$ such that 
\begin{equation} \label{4.2.4}
\varphi(y)=\left\{ \begin{array}{lr}
1, & y \in B_{d_0/3}(y_0),\\
0, & y \not\in B_{2d_0/3}(y_0).
\end{array}\right.
\end{equation}
We define 
$$\tilde R(y) = \varphi(y)R(y)
+ (1-\varphi(y))R(y_0), \quad 
\tilde B(y) = \varphi(y)B(y)
+ (1-\varphi(y))B(y_0).
$$
 In this section, we consider the 
following resolvent problem:
\begin{equation}\label{4.3}
\lambda  \tilde R \bv 
-\dv( \tilde B\nabla \bv) = \bff \quad \text{in $\Omega_+$},
\quad 
\tilde B(\nabla \bv \cdot \bn_+)
= \bg\quad\text{on $\Gamma_+$}.
\end{equation}
We shall prove the following theorem. 
\begin{thm}\label{thm:4.1} Let $1 < q \leq r$.  Let 
$X_q(\Omega_+)$ and $\CX_q(\Omega_+)$ be the spaces defined by
replacing $\Omega$ by $\Omega_+$ in Theorem \ref{thm:1.2}.
Then, there exist a small number $M_1 > 0$, a constant
$\lambda_0 > 0$ and an operator family $\CT_+(\lambda)$ with
$$\CT_+(\lambda) \in {\rm Hol}\,(\Sigma_{\epsilon, \lambda_0}, 
\CL(\CX_q(\Omega_+), H^2_q(\Omega)^n))$$
such that 
{such that if \eqref{4.2.1} is satisfied then} for any $\lambda \in \Sigma_{\epsilon, \lambda_0}$ and 
$(\bff, \bg) \in X_q(\Omega_+)$, $\bv = 
\CT_+(\lambda)(\bff, \lambda^{1/2}\bg, \bg)$ is a unique solution
of Eq. \eqref{4.3}, and 
$$\CR_{\CL(\CX_q(\Omega_+), H^{2-j}_q(\Omega_+)^n)}
(\{(\tau\pd_\tau)^\ell(\lambda^{j/2}\CT_+(\lambda)) \mid
\lambda \in \Sigma_{\epsilon, \lambda_0}\}) \leq r_b
$$
for $\ell=0,1$ and $j=0,1,2$ with some constant $r_b$
independent of $M_1$ and $M_2$, 
{where  $M_2$ is from \eqref{4.2.2}.}
\end{thm}
\begin{proof}
The uniqueness of 
solutions follows from the existence of solutions to the dual problem,
and so we only prove the existence of $\CR$ bounded solution operator
$\CT_+(\lambda)$.  We use the change of variables:
$y=\Phi(x)$ to transform Eq. \eqref{4.3} to the equations in the 
half-space.  We have
\begin{equation}\label{eq:4.4}
\Bigl(\frac{\pd x_j}{\pd y_k}\Bigr)(\Phi(x))
= a_{jk} + b_{jk}(x),
\end{equation}
where $a_{jk}$ and $b_{jk}(x)$ are the $(j,k)^{\rm th}$ components
of $\CA^{-1}$ and $\CB_{-1}(\Phi(x))$, respectively.  Since $\CA^{-1}$ is an orthogonal matrix and thanks to \eqref{4.1},
we have
\begin{equation}\label{4.5}
\sum_{j=1}^Na_{jk}a_{j\ell} = \sum_{j=1}^N
a_{kj}a_{\ell j} = \delta_{k\ell},
\quad \|b_{jk}\|_{L_\infty(\BR^N)} \leq M_1,
\quad\|\nabla b_{jk}\|_{L_r(\BR^N)} \leq C_K.
\end{equation}
By \eqref{eq:4.4}, we derive the formula for change of variables from $y$ to $x$, namely
\begin{equation}\label{4.4*}
\frac{\pd}{\pd y_j} = \sum_{k=1}^N (a_{kj} + b_{kj}(x))\frac{\pd}{\pd x_k}.
\end{equation}
Applying this formula  we get that 
\eq{
\frac{\pd^2}{\pd^2 y_j}=& \sum_{k,\ell=1}^N (a_{\ell j} + b_{\ell j}(x))\frac{\pd}{\pd x_\ell}\lr{(a_{kj} + b_{kj}(x))\frac{\pd}{\pd x_k}}\\
=&\sum_{k,\ell=1}^N(a_{\ell j}a_{kj} + b_{\ell j}(x)a_{kj})\frac{\pd^2}{\pd x_\ell\pd x_k}+\sum_{k,\ell=1}^N(a_{\ell j}b_{kj}(x) + b_{\ell j}(x)b_{kj}(x))\frac{\pd^2}{\pd x_\ell\pd x_k}\\
&+ \sum_{k,\ell=1}^N (a_{\ell j} + b_{\ell j}(x))\frac{\pd b_{kj}(x)}{\pd x_\ell}\frac{\pd}{\pd x_k}.
}
Note that by \eqref{4.5}, we have
$$\sum_{k,\ell=1}^N \sum_{j=1}^N a_{\ell j}a_{kj}\frac{\pd^2}{\pd x_\ell\pd x_k}=\sum_{k=1}^N\frac{\pd^2}{\pd^2 x_k},$$
therefore
\begin{equation}\label{4.4**}
\Delta_y = \sum_{j=1}^N \frac{\pd^2}{\pd^2y_j}
= \Delta_x+ D_2\nabla^2_x + D_1\nabla_x
\end{equation} where
\begin{align*}
\Delta_x &= \sum_{k=1}^N\frac{\pd^2}{\pd^2 x_k}, 
\quad
D_2\nabla_x^2 = \sum_{j,k,\ell=1}^N(a_{kj}b_{\ell j}(x) + 
b_{kj}(x)(a_{\ell j} + b_{\ell j}(x)))\frac{\pd^2}{\pd x_\ell \pd x_k}, 
\\
D_1\nabla_x & = \sum_{j,k,\ell=1}^N(a_{\ell j} + b_{\ell j}(x))
\frac{\pd b_{k j}(x)}{\pd x_j} \frac{\pd}{\pd x_k}.
\end{align*}
We now transform the form of the outer normal vector $\bn_+(y)$ to $\Gamma_+$ at point $y=\Phi(x)$. Since $\Gamma_+$ is represented by $x_N = \Phi^{-1}_N(y) = 0$, the gradient of function $\Phi^{-1}_N(y)$ will indicate the normal direction, therefore after normalization, we obtain
\begin{equation}\label{4.6}
\bn_+(y)=\bn_+(\Phi(x)) = -\frac{(\frac{\pd x_N}{\pd y_1}, \ldots, 
\frac{\pd x_N}{\pd y_N})^\top}{|(\frac{\pd x_N}{\pd y_1}, \ldots, 
\frac{\pd x_N}{\pd y_N})|}
= -\frac{(a_{N1} + b_{N1}(x), \ldots, a_{NN} + b_{NN}(x))^\top}{d(x)},
\end{equation}
where for the second equality we used \eqref{eq:4.4}.
Having this we note that
\eq{\label{4.13_new}
\nabla_y v^i(y)\cdot \bn_+(y)=\sum_{j=1}^N
\frac{\pd v^i(y)}{\pd y_j}n_+^j(y) 
= -\sum_{j,k=1}^N (a_{kj} + b_{kj}(x))\frac{a_{Nj}+b_{Nj}(x)}{d(x)}\frac{\pd  u^i(x)}{\pd x_k}\\
=-d^{-1}(x)\left[\frac{\pd u^i(x)}{\pd x_N}
+\sum_{j,k=1}^N\left\{(a_{kj} + b_{kj}(x))b_{Nj}(x)+a_{Nj}b_{kj}(x)\right\}\frac{\pd  u^i(x)}{\pd x_k}\right],
}
where we denoted $u^i(x)=v^i\circ\Phi(x)$.
Note that by \eqref{4.5} we have
$$d(x) = \sqrt{\sum_{j=1}^N(a_{Nj} + b_{Nj}(x))^2} = \sqrt{1 
+ \sum_{j=1}^N (2a_{Nj}b_{Nj}(x) + b_{Nj}(x)^2)}.$$ 
Therefore, choosing $M_1 > 0$ sufficiently small, we have
\begin{equation}\label{4.6*} 
d^{-1}(x)= 1 + \tilde d(x)
\end{equation}
with $|\tilde d(x)| \leq C|\sum_{j=1}^N (2a_{Nj}b_{Nj}(x) + b_{Nj}(x)^2)|\leq CM_1$ and 
\begin{align*}
\|\nabla \tilde d\|_{L_r(\BR^N)}&\leq C\sum_{j=1}^N(\|a_{Nj}\|_{L_\infty(\R^n)}+\|b_{Nj}\|_{L_\infty(\R^n)})\|\nabla b_{Nj}\|_{L_r(\R^n)} 
\leq CC_k\leq CM_2,
\end{align*} 
where in the last inequality we have used \eqref{4.2.3}.

Finally, by \eqref{4.4*}, \eqref{4.4**}, \eqref{4.6}, \eqref{4.13_new} and \eqref{4.6*}, the system
\eqref{4.3} is transformed to 
\begin{equation}\label{4.7}
\lambda R(y_0)\bu -
B(y_0)\Delta_x \bu + \bF(\bu)
= \tilde \bff \quad \text{in $\BR^N_+$}, \quad
B(y_0)(\nabla_x \bu\cdot\bn_0(x))+ \bG(\bu)= \tilde \bg
\quad\text{on $\BR^N_0$},
\end{equation}
where $\bn_0=(0,\ldots,0,-1)$, and
$$\bu(x) = \bv\circ \Phi(x), \quad \tilde \bff(x) = \bff\circ \Phi(x),\quad 
\tilde \bg(x) = \bg\circ \Phi(x),$$
and, by \eqref{4.2.4} and \eqref{4.6}-\eqref{4.6*}, 
\begin{align*}
\bF(\bu) & = 
\left\{\lambda[\varphi(\cdot)(R(\cdot)-R(y_0))\bv]
-[\dv_y(\varphi(y)(B(y)-B(y_0))\nabla_y\bv)]\right\}\circ\Phi\\
&\quad-B(y_0)(D_2\nabla_x^2\bu+D_1\nabla_x\bu), \\
\bG(\bu) & = B(y_0)\tilde d\nabla \bu \cdot\bn_0+\left\{\phi(y)(B(y)-B(y_0))\nabla_y \bv\cdot \bn_+\right\}\circ \Phi\\
&\quad- \frac{B(y_0)}{d}\sum_{j,k=1}^N(a_{kj}b_{Nj} + b_{kj}(a_{Nj}+b_{Nj}))
\frac{\pd\bu}{\pd x_k}.
\end{align*}

Using \eqref{sob:1}, \eqref{4.2.1}, \eqref{4.2.2}, and  \eqref{4.6*},
we have
\begin{equation}\label{4.9}\begin{aligned}
\|\bF(\bu)\|_{L_q(\BR^N_+)} & \leq C(|\lambda|M_1\|\bu\|_{L_q(\BR^N_+)}
+ (M_1+\alpha)\|\bu\|_{H^2_q(\BR^N_+)})
+ C_{\alpha, M_2}\|\bu\|_{H^1_q(\BR^N_+)}, \\ 
|\lambda|^{1/2}\|\bG(\bu)\|_{L_q(\BR^N_+)} & 
\leq CM_1|\lambda|^{1/2}\|\bu\|_{H^1_q(\BR^N_+)}, \\
\|\bG(\bu)\|_{H^1_q(\BR^N_+)} & \leq  C (M_1+\alpha)\|\bu\|_{H^2_q(\BR^N_+)}
+ C_{\alpha, M_2}\|\bu\|_{H^1_q(\BR^N_+)}
\end{aligned}\end{equation}
for any $\alpha > 0$, 
{where} $C$ is a constant independent of $\alpha$, 
$M_1$, $\lambda_1$ and $C_{\alpha, M_2}$ is a constant depending on 
$\alpha$ and $M_2$. 

Let $\CT_3(\lambda)$ be the $\CR$-bounded solution operator
for Eq. \eqref{3.1} given in Corollary \ref{cor:2.1}. Taking 
$\bu = \CT_3(\lambda)(\tilde\bff, \lambda^{1/2}\tilde\bg,
\bg)$ in \eqref{4.7}, we get 
\eq{\label{4.10}
&\lambda R(y_0)\CT_3(\lambda)(\tilde\bff, \lambda^{1/2}\tilde\bg,
\bg) -
B(y_0)\Delta_x \CT_3(\lambda)(\tilde\bff, \lambda^{1/2}\tilde\bg,\tilde\bg) + \bF(\CT_3(\lambda)(\tilde\bff, \lambda^{1/2}\tilde\bg,\tilde\bg))\\
&\hspace{7cm}= \tilde \bff + \bF(\CT_3(\lambda)(\tilde\bff, \lambda^{1/2}\tilde\bg,\tilde\bg))\quad \text{in $\BR^N_+$}, \\
&B(y_0)(\nabla_x \CT_3(\lambda)(\tilde\bff, \lambda^{1/2}\tilde\bg,\tilde\bg)\cdot\bn_0(x))+ \bG(\CT_3(\lambda)(\tilde\bff, \lambda^{1/2}\tilde\bg,\tilde\bg))\\
&\hspace{7cm}= \tilde \bg+\bG(\CT_3(\lambda)(\tilde\bff, \lambda^{1/2}\tilde\bg,\tilde\bg))
\quad\text{on $\BR^N_0$}.
}
Let us now denote
\begin{equation}
\CR_+(\lambda)H_\lambda(\tilde\bff,\tilde\bg)= ( \bF(\CT_3(\lambda)H_\lambda(\tilde\bff,\tilde\bg)), \bG(\CT_3(\lambda)H_\lambda(\tilde\bff,\tilde\bg))),
\end{equation}
where $H_\lambda(\tilde\bff,\tilde\bg)=(\tilde\bff, \lambda^{1/2}\tilde\bg,\tilde\bg)$.

By \eqref{4.9}, Corollary  \ref{cor:2.1} and Proposition \ref{prop:4.1}, we have
$$
\CR_{\CL(\CX_q(\BR^N_+))}(\{(\tau\pd_\tau)^\ell  H_\lambda\CR_+(\lambda) \mid
\lambda \in \Sigma_{\epsilon, \lambda_1}\}) \leq 
\{C(M_1 + \alpha) + C_{\alpha, M_2}\lambda_1^{-1/2} \}r_b
$$
for $\ell=0,1$. 
Thus, choosing $\alpha$ and $M_1$ so small that 
$C\alpha r_b < 1/8$, $CM_1 r_b < 1/8$ and choosing 
$\lambda_1 \geq \lambda_0$ so large that 
$C_{\alpha, M_2}\lambda_1^{-1/2}r_b \leq 1/4$, we have
\begin{equation} \label{4.12}
\CR_{\CL(\CX_q(\BR^N_+))}(\{(\tau\pd_\tau)^\ell H_\lambda \CR_+(\lambda) \mid
\lambda \in \Sigma_{\epsilon, \lambda_0}\}) \leq 1/2
\end{equation}
for $\ell=0,1$. Next, let us denote
\begin{equation} \label{4.13}
\bR_+(\lambda)(\tilde\bff, \tilde\bg) 
= \CR_+(\lambda)H_\lambda(\tilde\bff, \tilde\bg).
\end{equation}
Since for any $\lambda\not=0$ the norm $\|\tilde\bff,\tilde\bg\|_{X_q(\BR^N_+)}$ is 
equivalent to $\|H_\lambda(\tilde\bff,\tilde\bg)\|_{\CX_q(\BR^N_+)}$ (according to definition \eqref{norm:1}),
we can construct an operator 
$$(\bI +\bR_+(\lambda))^{-1} = \sum_{m=0}^\infty
(-\bR_+(\lambda))^m\quad \text{in} \quad X_q(\BR^N_+).$$
Rewriting now \eqref{4.10} as 
\begin{equation} \label{4.10a}
L(y_0)\CT_3(\lambda)H_\lambda(\tilde\bff, \tilde\bg) = [{\bf I} + \bR_+(\lambda)](\tilde\bff, \tilde\bg),
\end{equation}
with
\begin{equation}
L(y_0)(\cdot)= \left[ \begin{array}{c}
\lambda R(y_0)(\cdot) - B(y_0)\Delta(\cdot) +\bF(\cdot)\\
B(y_0)(\nabla(\cdot)\cdot\bn_0(x)) +\bG(\cdot) 
\end{array} \right]
\end{equation}
and taking 
$$
(\bar \bff, \bar \bg)=[{\bf I} + \bR_+(\lambda)](\tilde\bff, \tilde\bg) 
$$
in \eqref{4.10a} we see that 
$$
\bu = \CT_3(\lambda)H_\lambda(\bI - \bR_+(\lambda))^{-1}(\tilde\bff, \tilde\bg)
$$
$\bu \in H^2_q(\BR^N_+)^n$ is a unique solution of Eq. \eqref{4.7}. 

As for the $\CR$ bounded operator, the estimate \eqref{4.12} implies the existence 
of 
$$(\bI + H_\lambda\CR(\lambda))^{-1} = \sum_{m=0}^\infty
(-H_\lambda\CR(\lambda))^m.$$  By \eqref{4.13} we have
$$H_\lambda(\bI - \bR_+(\lambda))^{-1}
= (\bI - \CR(\lambda))^{-1}H_\lambda,$$
and so we have
\begin{equation}\label{4.14}
\bu = \CT_3(\lambda)(\bI - \CR(\lambda))^{-1}H_\lambda(\tilde\bff, \tilde\bg).
\end{equation}
Thus, setting 
$$\CT_4(\lambda) = \CT_3(\lambda)(\bI - \CR(\lambda))^{-1},$$ 
by \eqref{4.14}, \eqref{4.12}, and Corollary \ref{cor:2.1} we see that
$\bu = \CT_4(\lambda)H_\lambda(\tilde\bff, \tilde\bg)$ is a solution of
Eq. \eqref{4.7}, and 
$$\CR_{\CX_q(\BR^N_+), H^{2-j}_q(\BR^N_+))}
(\{(\tau\pd_\tau)^\ell(\lambda^{j/2}\CT_4(\lambda)) \mid
\lambda \in \Sigma_{\epsilon, \lambda_1} \})
\leq 2r_b
$$
for $\ell=0,1$ and $j=0,1,2$. If we set 
$$\CT_+(\lambda)F = [\CT_4(\lambda)F\circ\Phi]\circ\Phi^{-1},
$$
for $F = (F_1, F_2, F_3) \in \CX_q(\Omega_+)$, 
then, $\CT_+$ is a required $\CR$ bounded solution operator for 
Eq. \eqref{4.3}, which completes the proof of Theorem \ref{thm:4.1}.
\end{proof}

\section{Proof of Theorem \ref{thm:1.2}}\label{sec:5}
To prove Theorem \ref{thm:1.2}, we need to use several properties of
uniform $C^2$ domain, which are stated in the following proposition. For the proof of this result we refer for example to \cite{ES1}, Proposition 6.1.
\begin{prop}\label{prop:6.1} \mdseries Let $\Omega$ be a uniform 
$C^2$-domain in $\mathbb{R}^N$ with boundary $\Gamma$. 
Then,  for any positive constant $M_1$, there exist a constant
 $d \in (0, 1)$,  at most countably 
many functions $\Phi_j \in C^2(\mathbb{R}^N )$,  and points $x^1_j \in
\Omega$ and $x^2_j \in \Gamma$ $( j \in \BN)$   
such that the following assertions hold:
\begin{enumerate}
\item[\thetag1]~For every $j \in \BN$, the map $\BR^N \ni 
x \rightarrow \Phi_j(x) \in \BR^N$ is bijective.
\item[\thetag2]~$\Omega = (\bigcup_{j=1}^\infty B_d(x^1_j))
\cup (\bigcup_{j=1}^\infty (\Phi_j(\BR_+^N) \cap B_d(x^2_j)))$, 
$B_d(x^1_j) \subset \Omega$, 
$\Phi_j(\BR_+^N) \cap B_d(x^2_j) = \Omega \cap B_d(x^2_j)$, 
and $\Phi_j(\BR_0^N) \cap B_d(x^2_j) = \Gamma \cap B_d(x^2_j)$.
\item[\thetag3]~There exist $C^\infty$ functions $\zeta^i_j$, $\tilde\zeta^i_j$ $(i =  1, 2, j \in \BN)$ 
such that 
$${\rm supp}\, \zeta^i_j, \,{\rm supp}\,
\tilde{\zeta}^i_j \subset B_d(x^i_j),\qquad\|\zeta^i_j\|_{H^2_\infty(\mathbb{R}^N)} 
\leq c_0, \quad\| \tilde{\zeta}^i_j\|_{H^2_\infty(\mathbb{R}^N)} \leq c_0,$$ 
$$\tilde{\zeta}^i_j = 1\quad\text{on}\quad {\rm supp}\, \zeta^i_j,\qquad\sum_{i=1,2}\sum_{j=1}^\infty \zeta^i_j = 1\quad \text{on} \quad\overline{\Omega},\qquad
\sum_{j=1}^\infty \zeta^2_j = 1\quad \text{on} \quad\Gamma.$$ 
Here, $c_0$ is a constant which depends on $d$, $N$, $q$, $q'$ and $r$, 
but is independent of $j \in \BN$.
\item[\thetag4]~ 
$\nabla \Phi_j = \mathcal{R}_j + R_i, \nabla( \Phi_j)^{-1} 
= \mathcal{R}_j^{-} + R_j^{-} $, 
where $\mathcal{R}_j$ and $\mathcal{R}_j^{-}$ are 
$ N \times N$ constant orthogonal matrices, and $R_j$ and $R_j^{-}$ 
are $N\times N$ matrices of $H_\infty^1$ functions defined on $\BR^N$ 
which satisfy the conditions: 
$$\|R_j\|_{L_\infty(\BR^N)} \leq M_1, \quad \|R_j^{-}\|_{L_\infty(\BR^N)} \leq M_1,$$ 
and
$$\|\nabla R_j\|_{L_\infty(\BR^N)} \leq C_K,\quad\|\nabla R_j^-\|_{L_\infty(\mathbb{R}^N)} \leq C_K$$ 
for any $j \in \BN$. Here, $C_K$ is a constant depending only on constants 
$K$, $L_1$ and $L_2$  appearing in Definition \ref{dfn:1.2}. 
\item[\thetag5]~There exist a natural number $L>2$ such that
 any $L+1$ distinct sets of $\{B_d(x^i_j) \mid  i = 1, 2, 
\enskip j \in \BN\}$ 
have an empty intersection.
\end{enumerate}
\end{prop}
\bigskip

\noindent By the finite intersection property stated in point  \thetag5 of Proposition \ref{prop:6.1}, we have 
\begin{equation}\label{7.25}
\Bigl(\sum_{i=1, 2}
\sum_{j=1}^\infty \|f\|^q_{L_q(B^i_j\cap\Omega)}\Bigr)^{1/q} 
\leq C_q\|f\|_{L_q(\Omega)}
\end{equation}
for any $f \in L_q(\Omega)$ and $1 \leq q < \infty$. In particular, 
 by \eqref{7.25} we have
\begin{cor}\label{lem:cal} Let $i=1, 2$ and $1 < q < \infty$.
 Let $\{f_j\}_{j=0}^\infty$ be a sequence of functions
in $L_q(\Omega)$ such that $\sum_{j=0}^\infty \|f_j\|_{L_q(\Omega)}^q
< \infty$, and ${\rm supp}\, f_j \subset B_d(x^i_j)$ 
$(j \in \BN)$.
Then, $\sum_{j=0}^\infty f_j \in L_q(\Omega)$ and 
$\|\sum_{j=1}^\infty f_j\|_{L_q(\Omega)} \leq (\sum_{j=1}^\infty
\|f_j\|_{L_q(\Omega)}^q)^{1/q}$.
\end{cor}
In what follows, we write 
$\Omega_j = \Phi_j(\BR^N_+)$ 
and $\Gamma_j = \Phi_j(\BR^N_0)$
for $j \in \BN$. \\
Moreover, we denote  the 
unit outer normal to $\Gamma_j$ by $\bn_j$.
 Notice that $\bn_j= \bn$ on $\Gamma_j$.\\  
By \eqref{1.2}, choosing $d$ smaller
if necessary, we may assume that 
\begin{equation}\label{unif:1}
|R(x) - R(x^i_j)| \leq M_1, \quad 
|B(x)-B(x^i_j)|
\leq M_1
\quad \text{for $x \in B_d(x^i_j)\cap\overline{\Omega}$}.
\end{equation}
Let $\zeta^i_j$ and $\tilde \zeta^i_j$ be functions given in Proposition
\ref{prop:6.1} and set
$$R^{ij}(x) = \tilde \zeta^i_j(x) R(x)
+ (1-\tilde\zeta^i_j(x)) R(x^i_j), \quad
B^{ij}(x) = \tilde \zeta^i_j(x) B(x)
+ (1-\tilde\zeta^i_j(x)) B(x^i_j)
$$
Notice that
\begin{equation}\label{equal:1}
\zeta^i_j(x) R^{ij}(x) = \zeta^i_j(x)R(x),
\quad 
\zeta^i_j(x) B^{ij}(x) = \zeta^i_j(x)B(x), 
\end{equation}
because $\tilde \zeta^i_j = 1$ on ${\rm supp}\, \zeta^i_j$. 
To construct a parametrix for Eq. \eqref{1.8}, given 
$(\bff, \bg) \in X_q(\Omega)$, we consider the 
following equations:
\begin{gather}
 \lambda  R^{1j}\bv^1_j - \dv(
B^{1j}\nabla \bv^1_j) = \tilde\zeta^1_j\bff
\quad\text{in $\BR^N$}, \label{7.38} \\
\lambda  R^{2j}\bv^2_j - \dv(
B^{2j}\nabla \bv^2_1) = \tilde\zeta^2_j\bff \quad  \text{in $\Omega_j$}, \quad 
B^{2j}(\nabla \bv^2_j \cdot\bn_j)=\tilde\zeta^2_j\bg \quad 
\text{on $\Gamma_j$}.
\label{7.39}
\end{gather}
By Theorem \ref{thm:p.2} and Theorem \ref{thm:4.1}, there exist
$\CR$ bounded solution operators $\CD^i_j(\lambda)$ for Eq.
\eqref{7.38} and Eq. \eqref{7.39} 
with 
\begin{equation}\label{7.37}
\CD^1_j(\lambda) \in {\rm Hol}\,(\Sigma_{\epsilon, \lambda_0}, 
\CL(L_q(\BR^N)^n, H^2_q(\BR^N)^n)), \quad
\CD^2_j(\lambda) \in {\rm Hol}\,(\Sigma_{\epsilon, \lambda_0}, 
\CL(\CX_q(\Omega_j)^n, H^2_q(\Omega_j)^n))
\end{equation}
such that for any $(\bff, \bg) \in X_q(\Omega)^n$, 
$\bv^1_j = \CD^1_j(\lambda)\tilde\zeta^1_j\bff$ is a unique solution 
of Eq. \eqref{7.38}
and $\bv^2_j = \CD^2_j(\lambda)H_\lambda(\tilde\zeta^2_j\bff, \tilde\zeta^2_j \bg)$
is a unique solution of Eq. \eqref{7.39}, respectively.  Moreover,
we have
\begin{equation}\label{7.37.1}
\begin{aligned}
\CR_{\CL(L_q(\BR^N)^n, H^{2-k}_q(\BR^N)^n)}
(\{(\tau\pd_\tau)^\ell(\lambda^{k/2}\CD^1_j(\lambda)) \mid
\lambda \in \Sigma_{\epsilon, \lambda_0}\}) &\leq r_b, \\
\CR_{\CL(\CX_q(\Omega_j)^n, H^{2-k}_q(\Omega_j)^n)}
(\{(\tau\pd_\tau)^\ell(\lambda^{k/2}\CD^2_j(\lambda)) \mid
\lambda \in \Sigma_{\epsilon, \lambda_0}\}) &\leq r_b
\end{aligned}\end{equation}
for $\ell=0,1$ and $k=0,1,2$, where $\lambda_0$ and $r_b$ are
independent of $j \in \BN$. In particular, by \eqref{7.37.1}, 
we have
\begin{equation}\label{7.40*} \begin{aligned}
\sum_{k=0}^2|\lambda|^{k/2}\|\bv^1_j\|_{H^{2-k}_q(\BR^N)}
& \leq r_b\|\tilde\zeta^1_j\bff\|_{L_q(\BR^N)}, \\
\sum_{k=0}^2|\lambda|^{k/2}\|\bv^2_j\|_{H^{2-k}_q(\Omega_j)}
& \leq r_b\{\|\tilde\zeta^2_j\bff\|_{L_q(\Omega_j)} 
+ |\lambda|^{1/2}\|\tilde\zeta^2_j\bg\|_{L_q(\Omega_j)}
+ \|\tilde\zeta^2_j\bg\|_{H^1_q(\Omega_j)}\}.
\end{aligned}\end{equation}
Let us now introduce the notation
\begin{align*}
\bU(\lambda)(\bff, \bg)  = 
\sum_{i=1,2}\sum_{j=1}^\infty \zeta^i_j\bv^i_j, \quad 
\CU(\lambda)F  = \sum_{j=1}^\infty \zeta^1_j \CD^1_j(\lambda)F_1
+ \sum_{j=1}^\infty \zeta^2_j \CD^2_j(\lambda)F
\end{align*}
for $(\bff, \bg) \in X_q(\Omega)$ and $F = (F_1, F_2, F_3) \in 
\CX_q(\Omega)$.  By \eqref{7.25}, 
{Corollary} \ref{lem:cal},
\eqref{7.37} and \eqref{7.40*}, we have 
$\bU(\lambda)(\bff, \bg) \in H^2_q(\Omega)^N$, 
$\CU(\lambda) \in {\rm Hol}\,(\Sigma_{\epsilon, \lambda_0}, 
\CL(\CX_q(\Omega), H^2_q(\Omega)^n))$, 
\begin{equation}\label{7.41}\begin{aligned}
&\sum_{k=0}^2|\lambda|^{k/2}\|\bU(\lambda)(\bff, \bg)
\|_{H^{2-k}_q(\Omega)}
\leq C_qr_b(\|\bff\|_{L_q(\Omega)} + |\lambda|^{1/2}\|\bg\|_{L_q(\Omega)}
+ \|\bg\|_{H^1_q(\Omega)}), \\
&\CR_{\CL(\CX_q(\Omega), H^{2-k}_q(\Omega)^n)}
(\{(\tau\pd_\tau)^\ell(\lambda^{k/2}\CU(\lambda))\mid
\lambda \in \Sigma_{\epsilon, \lambda_0}\}) \leq C_qr_b.
\end{aligned}\end{equation}
Obviously, we have 
\begin{equation}\label{7.40}
\bU(\lambda)(\bff, \bg) = \CU(\lambda)H_\lambda(\bff, \bg).
\end{equation}
Moreover, noting \eqref{equal:1} and using \eqref{7.38} and 
\eqref{7.39}, we have
\begin{equation}\label{7.42}\left\{\begin{aligned}
\lambda R\bU(\lambda)(\bff, \bg) - \dv(B\nabla\bU(\lambda)(\bff, \bg)) 
= \bff - \bV_0(\lambda)(\bff, \bg)
&&&\quad\text{in $\Omega$}, \\
B(\nabla \bU(\lambda)(\bff, \bg) \cdot\bn) = \bg - \bV_b(\lambda)(\bff, \bg)
&&&\text{on $\Gamma$}.
\end{aligned}\right.\end{equation}
In the above we used the fact that 
$\tilde\zeta^i_j\zeta^i_j=\zeta^i_j$, $\sum_{i=1,2}\sum_{j=1}^\infty \zeta^i_j=1$, and so 
$$\sum_{i=1,2}\sum_{j=1}^\infty \tilde\zeta^i_j\zeta^i_j\bff=\sum_{i=1,2}\sum_{j=1}^\infty \zeta^i_j\bff=\bff,$$
and we denoted
\begin{align*}
\bV_0(\lambda)(\bff, \bg) & = \sum_{i=1,2}\sum_{j=1}^\infty \dv(B^{ij}
(\nabla\zeta^i_j)\bv^i_j) 
+ \sum_{i=1,2}\sum_{j=1}^\infty(\nabla\zeta^i_j)
\cdot(B^{ij}\nabla \bv^i_j), \\
\bV_b(\lambda)(\bff, \bg) & = \sum_{j=1}^\infty B^{2j}
(\nabla\zeta^2_j\cdot\bn_j)\bv^2_j.
\end{align*}
Let us also denote
\begin{align*}
\CV_0(\lambda)F & = \sum_{j=1}^\infty \dv(B^{1j} (\nabla\zeta^1_j)\CD^1_j(\lambda)(\tilde\zeta^1_jF_1) )
+ \sum_{j=1}^\infty B^{2j}(\nabla\zeta^2_j)\CD^2_j(\lambda)(\tilde\zeta^2_jF)\\ 
&+ \sum_{j=1}^\infty(\nabla\zeta^1_j)
\cdot(B^{1j}\nabla \CD^1_j(\lambda)(\tilde\zeta^1_jF_1))
+ \sum_{j=1}^\infty(\nabla\zeta^2_j)
\cdot(B^{2j}\nabla \CD^2_j(\lambda)(\tilde\zeta^2_jF)), \\
\CV_b(\lambda)F & = \sum_{j=1}^\infty B^{2j}
(\nabla\zeta^2_j\cdot\bn_j)\CD^2_j(\lambda)(\tilde\zeta^2_jF),
\end{align*}
for $F=(F_1, F_2, F_3) \in \CX_q(\Omega)$.  Moreover, we set
$$\bV(\lambda)(\bff, \bg) 
= (\bV_0(\lambda)(\bff, \bg), \bV_b(\lambda)(\bff, \bg)),
\quad 
\CV(\lambda)F = (\CV_0(\lambda)F, \CV_b(\lambda)F).
$$
In particular, we have
\begin{equation}\label{equal:2}
\bV(\lambda)(\bff, \bg) = \CV(\lambda)H_\lambda(\bff, \bg)
\end{equation}
for any $(\bff, \bg) \in X_q(\Omega)$. By Proposition \ref{prop:4.1}, 
\eqref{7.37.1}, \eqref{sob:1} and  
\eqref{1.2}, we have
$$\CR_{\CL(\CX_q(\Omega))}(\{(\tau\pd_\tau)^\ell(H_\lambda\CV(\lambda)) \mid 
\lambda \in \Sigma_{\epsilon, \lambda_1}\}) \leq CM_0r_b\lambda^{-1/2}_1
$$
for $\ell=0,1$ and $\lambda_1 \geq \lambda_0$.  Thus, choosing 
$\lambda_0$ so large that $CM_0r_b\lambda_1^{-1/2} \leq 1/2$, we have
\begin{equation}\label{7.43}
\CR_{\CL(\CX_q(\Omega))}(\{(\tau\pd_\tau)^\ell(H_\lambda\CV(\lambda)) \mid 
\lambda \in \Sigma_{\epsilon, \lambda_0}\}) \leq 1/2
\end{equation}
for $\ell=0,1$.  By \eqref{equal:2} and \eqref{7.43}, we have
\begin{equation}\label{cont:1}
\|H_\lambda \bV(\lambda)(\bff, \bg)\|_{\CX_q(\Omega)}
\leq (1/2)\|H_\lambda(\bff, \bg)\|_{\CX_q(\Omega)}.
\end{equation}
The $\|H_\lambda(\bff, \bg)\|_{\CX_q(\Omega)}$ is equivalent norm to 
$\|(\bff, \bg)\|_{X_q(\Omega)}$ for $\lambda\not=0$, and therefore,
it follows from \eqref{cont:1} that the 
inverse operator $(\bI - \bV(\lambda))^{-1} = \sum_{j=0}^\infty \bV(\lambda)^j$
exists in $X_q(\Omega)$. Moreover, by \eqref{7.43}, the inverse operator
$(\bI-H_\lambda\CV(\lambda))^{-1} 
= \sum_{j=0}^\infty (H_\lambda\CV(\lambda))^j$ exists in 
$\CX_q(\Omega)$. By \eqref{equal:2}, 
\begin{equation}\label{equal:3}
H_\lambda(\bI - \bV(\lambda))^{-1} = (\bI-H_\lambda\CV(\lambda))^{-1}H_\lambda.
\end{equation}
In view of  \eqref{7.42} and \eqref{7.40}, 
$\bv = \bU(\lambda)(\bI - \bV(\lambda))^{-1}(\bff, \bg)$ 
is a unique solution of Eq. \eqref{1.1*} or \eqref{1.1}. 
{The uniqueness}
follows from the existence of the dual problem. By \eqref{7.40}
and \eqref{equal:3}, this $\bv$ is represented by $\bv
= \CU(\lambda)(\bI-\CH \CV(\lambda))^{-1}H_\lambda(\bff, \bg)$.  Thus, 
setting $\CS(\lambda) = \CU(\lambda)(\bI-H_\lambda \CV(\lambda))^{-1}$, 
by \eqref{7.40}, \eqref{7.42} and Proposition \ref{prop:4.1}, 
we see that $\bv = \CS(\lambda)H_\lambda(\bff, \bg)$ is a unique solution
of Eq. \eqref{1.1} and 
$$\CR_{\CL(\CX_q(\Omega), H^{2-k}_q(\Omega)^n)}
(\{(\tau\pd_\tau)^\ell(\lambda^{k/2}\CS(\lambda)) \mid 
\lambda \in \Sigma_{\epsilon, \lambda_0}\}) \leq 2r_b
$$
for $\ell=0,1$ and $k=0,1,2$.  This completes the proof of 
Theorem \ref{thm:1.2}. \hfill
\begin{minipage}[t]{1cm}
\flushright $\Box$
\end{minipage}
\section{Proof of Theorem \ref{thm:1.1}} \label{sec:6}
To prove the existence part of Theorem \ref{thm:1.1}, we first consider an artificial
 initial-boundary  problem: 
\begin{equation}\label{8.1}
\pd_t\bu - R^{-1}\dv(B\nabla\bu) = 0\quad \text{in $\Omega\times(0, T)$}, 
\quad
B(\nabla\bu\cdot\bn)|_\Gamma=0, 
\quad \bu|_{t=0} = \bu_0.
\end{equation}
The corresponding resolvent problem of Eq. \eqref{8.1} is the following
system:
\begin{equation}\label{8.2}
\lambda \bv - R^{-1}\dv(B\nabla\bv) = \bff \quad\text{in $\Omega$},
\quad B(\nabla\bv\cdot\bn)|_\Gamma = 0.
\end{equation}
If we set
\begin{align*}
\bD_q(\Omega) &= \{\bv \in H^2_q(\Omega)^n \mid
B(\nabla\bv\cdot\bn) = 0 \quad\text{on $\Gamma$}\}, \\
\bA\bv &= R^{-1}\dv(B\nabla\bv) \quad\text{for $\bv \in \bD_q(\Omega)$},
\end{align*}
then Eq. \eqref{8.2} is written in the form:
\begin{equation}\label{8.3}
(\lambda-\bA)\bv = \bff.
\end{equation}
Let $\CS(\lambda)$ be the $\CR$-bounded solution operator given in 
Theorem \ref{thm:1.2}, then a unique solution of 
\eqref{8.3} is given by $\bv = \CS(\lambda)(R\bff, 0)$.  Therefore, by Theorem
\ref{thm:1.2} and \eqref{1.3}, we have
$$\sum_{k=0}^2|\lambda|^{k/2}\|\bv\|_{H^{2-k}_q(\Omega)}
\leq C_{m_1}r_b\|\bff\|_{L_q(\Omega)},
$$
for any $\lambda \in \Sigma_{\epsilon, \lambda_0}$ and 
$\bff \in L_q(\Omega)^n$. 
By the semi-group theory, the operator $\bA$ generates an $C_0$ analytic
semigroup $\{T(t)\}_{t\geq0}$ possessing the estimate:
\begin{gather*}
\|T(t)\bu_0\|_{L_q(\Omega)}+t\|\pd_tT(t)\bu_0\|_{L_q(\Omega)}
\leq Ce^{\gamma t}\|\bu_0\|_{L_q(\Omega)}, \\
\|\pd_tT(t)\bu_0\|_{L_q(\Omega)} \leq Ce^{\gamma t}\|\bu_0\|_{H^2_q(\Omega)},
\end{gather*} 
for any $t > 0$ with some constants $\gamma \in \BR$ and $C>0$. 
Using the real interpolation theorem (cf. Tanabe \cite[Subsec. 1.4]{Tanabe})
we can prove:
\begin{thm}\label{thm:semi} Let $1 < p, q < \infty$.  Assume that 
$\Omega$ is a uniformly $C^2$ domain.  Let 
$$\CD_{q,p}(\Omega) = (L_q(\Omega)^n, \bD_q(\Omega))_{1-1/p,p},$$
where $(\cdot, \cdot)_{1-1/p,p}$ is a real interpolation functor (\cite[Chapter 7]{Ad}). Then,
for any $\bu_0 \in \CD_{p,q}(\Omega)$, problem \eqref{8.1} admits
a unique solution $\bu$ with 
$$e^{-\gamma t}\bu \in H^1_p((0, \infty), L_q(\Omega)^n)
\cap L_p((0, \infty), H^2_q(\Omega)^n)$$
possessing the esitmate:
$$\|e^{-\gamma t}\bu\|_{L_p((0, \infty), H^2_q(\Omega))}
+ \|e^{-\gamma t}\pd_t\bu\|_{L_p((0, \infty), L_q(\Omega))}
\leq C\|\bu_0\|_{B^{(2-1/p)}_{q,p}(\Omega)}
$$
for any $\gamma > \lambda_0$ with some constant $C$ depending
on $\lambda_0$ that is the same as in Theorem \ref{thm:1.2}.
\end{thm}

\begin{proof} The proof of Theorem \ref{thm:semi} follows the same lines as the Theorem 3.9 in \cite{SS0}, so we skip it.
\end{proof}
\begin{remark} Note that $\bu_0 \in B^{2(1-1/p)}_{q,p}(\Omega)^n$ satisfies
the condition:
$$B(\nabla\bu_0\cdot\bn) = 0 \quad\text{on $\Gamma$}, $$
then $\bu_0 \in \CD_{q,p}(\Omega)$ when $2/p + 1/q < 1$. 
Moreover, when $2/p + 1/q > 1$, than any $\bu_0 \in B^{2(1-1/p)}_{q,p}(\Omega)$
belongs to $\CD_{q,p}(\Omega)^n$. 
\end{remark}

We now proceed the existence part of Theorem \ref{thm:1.1}. Let 
$\CS(\lambda) \in {\rm Hol}\,(\Sigma_{\epsilon, \lambda_0}, 
\CL(\CX_q(\Omega), H^2_q(\Omega)^n))$ be a solution operator of problem
\eqref{1.8} that exists due to  Theorem \ref{thm:1.2}.   
Let 
$$\bF \in L_p((0, T), L_q(\Omega)^n), 
\quad e^{-\gamma t}\bG \in L_p(\BR, H^1_q(\Omega)^n)
\cap H^{1/2}_p(\BR, L_q(\Omega)^n).
$$
for any $\gamma > \lambda_0$. Let  
$\bF_0$ be the zero extension of $\bF$ outside of $(0, T)$, that is
$\bF_0(\cdot, t) = \bF(\cdot, t)$ for $t \in (0, T)$ and $\bF_0(\cdot, t) = 0$
for $t \not\in (0, T)$. We consider the following time-dependent problem:
\begin{equation}\label{8.5}
R\pd_t\bv - \dv(B\nabla\bv) = \bF_0 \quad\text{in $\Omega\times\BR$},\quad
B(\nabla\bv\cdot\bn) = \bG \quad\text{on $\Gamma\times\BR$}.
\end{equation}
Let $\CL$ and $\CL^{-1}$ be the Laplace transform and the Laplace inverse
transform, that is
\begin{align*}
\CL[f](\lambda) &= \int^\infty_{-\infty}e^{-(\gamma +i\tau)t}
f(t)\,dt = \CF[e^{-\gamma t}f](\tau) \quad(\lambda=\gamma + i\tau),
\\
\CL^{-1}[g](t) &= \frac{1}{2\pi}\int^\infty_{-\infty}
e^{\gamma + i\tau}tg(\gamma + i\tau)\,d\tau
= e^{\gamma t}\CF^{-1}_\tau[g(\gamma + i\tau)](t).
\end{align*}
Applying Laplace transformation to  \eqref{8.5}, we have
$$
\lambda R\CL[\bv] - \dv(B\nabla\CL[\bv]) = 
\CL[\bF_0] \quad\text{in $\Omega$},\quad
B(\nabla\CL[\bv]\cdot\bn) = \CL[\bG] \quad\text{on $\Gamma$}.
$$
In view of Theorem \ref{thm:1.2}, we have
$$\CL[\bv] = \CS(\lambda)(\CL[\bF_0](\lambda), \lambda^{1/2}\CL[\bG](\lambda),
\CL[\bG](\lambda))$$
for $\gamma > \lambda_0$ with $\lambda = \gamma + i\tau \in \BC$. 
Thus, a solution $\bv$ of Eq. \eqref{8.5} is given by 
\begin{align*}
\bv &= \CL^{-1}[\CS(\lambda)(\CL[\bF_0](\lambda), 
\lambda^{1/2}\CL[\bG](\lambda), \CL[\bG](\lambda))](t) \\
&=e^{\gamma t}\CF_\tau^{-1}[\CS(\gamma + i\tau)\CF[e^{-\gamma t}(\bF_0,
\Lambda_\gamma^{1/2}\bG, \bG)](\tau)](t)
\end{align*}
for any $\gamma > \lambda_0$.  Here, $\Lambda^{1/2}_\gamma$ is the operator
defined by setting
$$\Lambda^{1/2}_\gamma g = \CL^{-1}[\lambda^{1/2}\CL[g](\lambda)].
$$
By the Cauchy theorem in theory of functions
of one complex variable, the value of $\bv$ is independent of choice of
$\gamma > \lambda_0$. By Theorem \ref{thm:1.2} and
Weis's operator valued Fourier multiplier theorem \cite{Weis}, 
we have
{\begin{align*}
&\|e^{-\gamma t}\bv\|_{L_p((\BR, H^2_q(\Omega))} + 
\|e^{-\gamma t}\pd_t\bv\|_{L_p(\BR, L_q(\Omega))}\\
&\quad \leq C(\|e^{-\gamma t}\bF_0\|_{L_p(\BR, L_q(\Omega))}
+ \|e^{-\gamma t}\Lambda_\gamma^{1/2}\bG\|_{L_p(\BR, L_q(\Omega))}
+ \|e^{-\gamma t}\bG\|_{L_p(\BR, H^1_q(\Omega))})
\end{align*}}
for any $\gamma > \lambda_0$ with some constant $C$ depending on $\lambda_0$. 
Since $|(\tau\pd_\tau)\lambda^{1/2}(1+\tau^2)^{-1/4}| \leq C(1+\gamma^{1/2}$
for any $\lambda = \gamma + i\tau \in \BC$ with $\gamma > \lambda_0$, 
by Proposition \ref{prop:4.1} we have
$$\|e^{-\gamma t}\Lambda^{1/2}_\gamma \bG\|_{L_p(\BR, L_q(\Omega))}
\leq C(1+\gamma^{1/2})\|e^{-\gamma t}\bG\|_{H^{1/2}_p(\BR, L_q(\Omega))}.
$$
Summing up, we have proved that $\bv$ satisfies Eq. \eqref{8.5}
and the estimate:
{\begin{align*}
&\|\bv\|_{L_p(((0, T), H^2_q(\Omega))} + 
\|\pd_t\bv\|_{L_p((0, T), L_q(\Omega))}\\
&\quad \leq Ce^{\gamma T}(\|\bF\|_{L_p((0, T), L_q(\Omega))}
+(1+\gamma^{1/2})
 \|e^{-\gamma t}\bG\|_{H^{1/2}_p(\BR, L_q(\Omega))}
+ \|e^{-\gamma t}\bG\|_{L_p(\BR, H^1_q(\Omega))})
\end{align*}}
for any $\gamma > \lambda_0$ with some constants $C$ depending 
on $\lambda_0$.

Next, to compensate for the lack of the  initial condition, we consider the following initial problem:
\begin{equation}\label{8.9}
R\pd_t\bw - \dv(B\nabla\bw) = 0 \quad\text{in $\Omega\times(0, \infty)$},
\quad B(\nabla\bw\cdot\bn)|_\Gamma=0, \quad
\bw|_{t=0} = \bu_0-\bv|_{t=0}.
\end{equation}
By \eqref{1.4}, we see that $\bu_0 - \bv|_{t=0} \in \CD_{q,p}(\Omega)$
when $2/p+ 1/q \not=1$, and so, by Theorem \ref{thm:semi}, 
problem \eqref{8.9} admits a unique solution $\bw$ with 
$$e^{-\gamma t}\bw \in H^1_p((0, \infty), L_q(\Omega)^n) \cap
L_p((0, \infty), H^2_q(\Omega)^n)$$
possessing the estimate:
$$\|e^{-\gamma t}\bw\|_{L_p((0, \infty), H^2_q(\Omega))}
+ \|e^{-\gamma t}\pd_t\bw\|_{L_p((0, \infty), L_q(\Omega))}
\leq C\|\bu_0-\bv|_{t=0}\|_{B^{2(1-1/p)}_{q,p}(\Omega))},
$$ 
for any $\gamma >\lambda_0$. 
Again, by the real interpolation theorem we have
{$$\|\bv|_{t=0}\|_{B^{2(1-1/p)}_{q,p}(\Omega)}
\leq C(\|e^{-\gamma t}\bv
\|_{L_p((0, \infty), H^2_q(\Omega))} 
+ \|e^{-\gamma t}\pd_t\bv\|_{L_p((0, \infty), L_q(\Omega))})
$$}
for some $\gamma > \lambda_0$, because $e^{-\gamma t}\bv|_{t=0}
= \bv|_{t=0}$.  

Summing up, we have proved that $\bu = \bv + \bw$ is a required 
solution of Eq. \eqref{1.1} or equivalently of \eqref{1.1*} 
possessing the estimate \eqref{1.6}. This completes the proof of
of the first part of  Theorem \ref{thm:1.1} devoted to the existence of a solution.

In order to prove the uniqueness of solutions of Eq. \eqref{1.1} we now consider
$\bu$ satisfying the regularity condition \eqref{1.5} and the homogeneous  system of equations \eqref{homo:1}.
Let $\bu_0$ be the zero extension of $\bu$ to $t < 0$, that is 
$\bu_0(\cdot, t) = \bu(\cdot, t)$ for $t \in (0, T)$ and $\bu_0(\cdot, t) = 0$
for $t < 0$.  We define $\bv$ by letting
$$\bv(\cdot, t) = \begin{cases} \bu_0(\cdot, t)&\quad\text{for $t < T$} \\
\bu_0(\cdot, 2T-t)&\quad\text{for $t \geq T$}.
\end{cases}
$$
Since $\bu|_{t=0} = 0$, we see that 
$$\bv \in H^1_p(\BR, L_q(\Omega)^n) \cap L_p(\BR, H^2_q(\Omega)^n), $$
that $\bv$ vanishes for $t \not\in (0, 2T)$, and that $\bv$ satisfies
the homogeneous equations:
\begin{equation}\label{homo:3}
R\pd_t\bv- \dv(B\nabla\bv) = 0\quad\text{in $\Omega\times\bR$}, \quad
B(\nabla\bv\cdot\bn)|_{\Gamma}=0.
\end{equation}
Applying the Laplace transform to \eqref{homo:3} yields that
$$
\lambda R\CL[\bv]- \dv(B\nabla\CL[\bv]) = 0\quad\text{in $\Omega$}, \quad
B(\nabla\CL[\bv]\cdot\bn)|_{\Gamma}=0.
$$
Since
\begin{align*}
\|\CL[\bv](\gamma + i\tau)\|_{H^2_q(\Omega)} & \leq \int^{2T}_0e^{\gamma t}
\|\bv(\cdot, t)\|_{H^2_q(\Omega)}\,dt
\leq e^{2\gamma T}(2T)^{1/{p'}}\|\bv\|_{L_p((0, 2T), H^2_q(\Omega))} \\
&\leq 2e^{\gamma T}(2T)^{1/{p'}}\|\bu\|_{L_p((0, T), H^2_q(\Omega))}
< \infty,
\end{align*}
the uniqueness stated in Theorem \ref{thm:1.2} yields that $\CL[\bv](\lambda) 
= 0$ for $\lambda \in \Sigma_{\epsilon, \lambda_0}$.  But, $\CL[\bv](\lambda)$
is holomorphic in $\BC$, because $\bv$ vanishes for $t \not\in (0, 2T)$.
Thus, $\CL[\bv]$ is identically zero, which yields that $\bv=0$.  Thus,
$\bu = 0$.  This completes the proof of uniqueness of solutions from Theorem
\ref{thm:1.1}. 
\hfill
\begin{minipage}[t]{1cm}
\flushright $\Box$
\end{minipage}

\section{{Proof of Theorem \ref{thm:1.1global}}} \label{sec:7}
We follow an argument from Section 3 of \cite{Shibata18}. 
First we prove the exponential stability
of semigroup $\{T(t)\}_{t\geq 0}$ associated with the problem 
\begin{equation}\label{eq:1.1}
R\pd_t\bu - \dv(B\nabla\bu) = 
0 \quad \text{in $\Omega\times(0, \infty)$}, \quad 
B(\nabla\bu\cdot\bn)|_\Gamma=0, \quad \bu|_{t=0} = \bu_0.
\end{equation}
For this purpose, we consider the resolvent problem:
\begin{equation}\label{e:1}
\lambda R\bv - \dv(B\nabla\bv) = \bff \quad\text{in $\Omega$}, 
\quad B(\nabla\bv\cdot\bn)|_\Gamma=0.
\end{equation}
Let us define:
\begin{align*}
\hat L_q(\Omega)^n & = \{\bff \in L_q(\Omega)^n \mid \int_\Omega\bff\,dx = 0\}, \\
\hat H^2_q(\Omega)^n & = \{\bv \in H^2_q(\Omega)^n \mid 
B(\nabla\bv\cdot\bn)|_\Gamma=0, \quad 
\int_\Omega R\bv\,dx = 0\}.
\end{align*}
By Theorem \ref{thm:1.2}, there exists a $\lambda_0 > 0$ such that 
for any $\lambda \in \Sigma_{\epsilon, \lambda_0}$ and $\bff \in 
L_q(\Omega)^n$, problem \eqref{e:1} admits a unique solution
$\bv \in H^2_q(\Omega)^n$ satisfying:
\begin{equation}\label{resol:1}
|\lambda|\|\bv\|_{L_q(\Omega)} + \|\bv\|_{H^2_q(\Omega)}
\leq C\|\bff\|_{L_q(\Omega)}
\end{equation}
for some constant $C > 0$. In addition, if $\bff \in \hat L_q(\Omega)^n$, then 
$\bv \in \hat H^2_q(\Omega)^n$ when $\lambda\not=0$.  In fact, integrating 
\eqref{e:1} and using the Gauss divergence theorem leads to
$$\lambda \int_\Omega R\bv\,dx = 0,
$$
which, combined with $\lambda\not=0$, yields that 
\begin{equation}\label{e:2}
\int_\Omega R\bv\,dx = 0.
\end{equation}
Let $\CB$ be an operator acting on $\bv \in \hat H^2_q(\Omega)^n$
defined by setting $\CB\bv = \dv(B\nabla\bv)$ for 
$\bv \in \hat H^2_q(\Omega)^n$. Then $(\lambda R- \CB)$ 
is a bijective map from $\hat H^2_q(\Omega)^n$
onto $\hat L_q(\Omega)^n$ when $\lambda \in \Sigma_{\epsilon, \lambda_0}$.
Since $\Omega$ is bounded, by the Rellich compactness theorem 
$(\lambda R-\CB)^{-1}$ is a compact operator 
from $L_q(\Omega)^n$ into itself.  Thus, by Riesz-Schauder theory,
especially Fredholm alternative principle, the injectiveness of
$\lambda R - \CB$ implies the bijectiveness. Let $\lambda \not\in(-\infty, 0)$
and let $\bv \in \hat H^2_q(\Omega)^n$ satisfy the homogeneous equations:
$$\lambda R\bv-\dv(B\nabla\bv) = 0\quad\text{in $\Omega$}, \quad 
B(\nabla\bv\cdot\bn)|_\Gamma=0.$$
Let $2 \leq q < \infty$, and then $\hat H^2_q(\Omega) \subset \hat H^2_2(\Omega)$.
Multiplying the above equation by $\overline{\bv}$, with
$\overline{\bv}$ being the complex conjugate of $\bv$,  integrating
the resulting formula over $\Omega$, and using
the Gauss divergence theorem leads to
\begin{equation} \label{7:1}
\lambda(R\bv, \overline{\bv})_\Omega + (B\nabla\bv, \overline{\nabla\bv})_\Omega=0, 
\end{equation}
where 
$$
(B\nabla\bv, \overline{\nabla\bv}):=\sum_{k,l=1}^n B_{kl}(\nabla v_l,\nabla \overline{v_k})_{\Omega}
=\sum_{j,k,l=1}^n(B_kl\de_{x_j}v_l,\de_{x_j}\overline{v_k})_{\Omega}
=\sum_{j=1}^n(B \de_{x_j}\bv,\overline{\de_{x_j}\bv})_{\Omega}.
$$
In particular, $(R\bv, \overline{\bv})_\Omega$ and $(B\nabla\bv, \overline{\nabla\bv})_\Omega$ 
are real numbers. Therefore, if ${\rm Im}\,\lambda \neq 0$, taking the imaginary part of \eqref{7:1} 
we have $(R\bv,\overline{\bv})_{\Omega}=0$ which yields that 
$\|\bv\|_{L_2(\Omega)}^2 = 0$. Thus, we have $\bv=0$, that is the 
uniqueness holds. In ${\rm Im}\,\lambda = 0$ then  ${\rm Re}\,\lambda \geq 0$ since 
$\lambda \not\in(-\infty, 0)$. Now in order to show uniqueness we take the real part of \eqref{7:1} 
which implies 
$$
m_1\|v\|_{L_2(\Omega)}^2+\|\nabla \bv\|_{L_2(\Omega)}^2 \leq 0.
$$
Thus, again, $\bv=0$. 
From these considerations, for $\lambda\not\in(-\infty,0)$, $(\lambda R- \CB)$
is a bijective map from $\hat H^2_q(\Omega)^n$ onto $\hat L_q(\Omega)^n$ 
provided $2 \leq q < \infty$.  In the case where $1 < q < 2$, the uniqueness
follows from the bijectiveness 
of the operator $\bar\lambda R-\CB$ for
$2 \leq q < \infty$, and so the operator $(\lambda R- \CB)$ is also a 
bijective map from $\hat H^2_q(\Omega)^n$ onto 
$\hat L_q(\Omega)^n$. 
From the standard argument
in the theory of $C_0$ analytic semigroups,
we see that for any $\epsilon \in (0, \pi/2)$ 
the resolvent estimate \eqref{resol:1} holds for any
$\lambda \in \Sigma_\epsilon\cup\{0\}$ 
with some uniform constant $C$
depending solely on $\epsilon$. From this it follows that 
there exists a $C_0$ analytic semigroup $\{T(t)\}_{t\geq 0}$ associated with
problem \eqref{eq:1.1} possessing the estimate:
\begin{equation}\label{exp:1}
\|T(t)\bu_0\|_{L_q(\Omega)} \leq Me^{-\delta t}\|\bu_0\|_{L_q(\Omega)},
\end{equation}
for any $t > 0$ and $\bu_0 \in \hat L_q(\Omega)^n$ with some positive constants
$M$ and $\delta$. 

We now prove Theorem \ref{thm:1.1global}.  
For this purpose, we first consider the shifted equations:
\begin{equation}\label{e:6}\begin{aligned} 
 R(\pd_t\bw + \eta \bw) -\dv B(\nabla\bw) &= \bF
&\quad&\text{in $\Omega\times(0, \infty)$}, \\
B(\nabla\bw\cdot\bn) &=\bG &\quad&\text{on $\Gamma \times(0, \infty)$}, \\
\bw|_{t=0} & = \bu_0 &\quad&\text{in $\Omega$}.
\end{aligned}\end{equation}
In view of Theorem \ref{thm:1.2}, there exist a large positive constant $\eta$ and a
positive constant $\gamma_0$ such that any solution 
$\bw$ of equations \eqref{e:6} satisfies the exponential decay
property:
\begin{equation}\label{e:8}
\|e^{\gamma t}\bw\|_{L_p((0, \infty), H^2_q(\Omega))} + \|e^{\gamma t}
\pd_t\bw\|_{L_p((0, \infty), L_q(\Omega))} 
\leq C\CF_\gamma
\end{equation}
for any $\gamma \leq \gamma_0$ with some positive constants $C > 0$ 
and $\gamma_0$, where we have set 
\begin{align*}
\CF_\gamma & = \|\bu_0\|_{B^{2(1-1/p)}_{q,p}(\Omega)} + \|e^{\gamma t}
\bF\|_{L_p((0, \infty), L_q(\Omega))}
+ \|e^{\gamma t}\bG\|_{L_p(\BR, H^1_q(\Omega))} 
+ (1+\gamma^{1/2})\|e^{\gamma t}\bG\|_{H^{1/2}_p(\BR, L_q(\Omega))}.
\end{align*}
In fact, $\Sigma_\epsilon + \eta = \{\lambda + \eta \mid 
\lambda
\in \Sigma_\epsilon\}
\subset \Sigma_{\epsilon, \lambda_0}$ for any large $\eta > 0$. Repeating the 
proof of Theorem \ref{thm:1.1} gives the assertion above. 

In particular, conditions \eqref{cond:1} and \eqref{cond:2} give that 
\begin{equation}\label{orth:1}
\int_\Omega R(x)\bw(x, t)\,dx = 0 \quad\text{for any $t > 0$}.
\end{equation}
In fact, integrating \eqref{e:6} over $\Omega$ and using the Gauss divergence
theorem implies that
$$\frac{d}{dt}\int_\Omega R\bw\,dx +\eta\int_{\Omega}R\bw\,dx = \int_\Omega \bF(x, t)\,dx + \int_\Gamma
\bG(x, t)\,d\sigma = 0
$$ for any $t > 0$ because of \eqref{cond:1}.  Integrating this formula
over $(0, t)$ and using \eqref{cond:1} give that
$$\int_\Omega R(x)\bw(x, t)\,dx = \int_\Omega R(x)\bu_0(x)\,dx = 0 \quad\text{for any $t > 0$}.
$$
We now consider the compensation equation:
$$
R\pd_t\bv - \dv(B\nabla\bv) = -\eta R\bw 
\quad\text{in $\Omega\times(0, \infty)$}, \quad
B(\nabla\bv\cdot\bn)|_\Gamma=0, \quad \bv|_{t=0}.
$$
Since $\bw(x, t) \in \hat L_q(\Omega)^n$ for any $t > 0$
as follows from \eqref{orth:1}, 
by the Duhamel principle, we have
$$\bv(\cdot, t) = -\eta\int^t_0T(t-s)(R\bw)(\cdot, s)\,ds.
$$
Choosing $\gamma_0$ smaller if necessary, we may assume that $\delta
> \gamma_0$, and so by \eqref{exp:1} 
\begin{align*}
\|e^{\gamma t}\bv(\cdot, t)\|_{L_q(\Omega)} & \leq 
M\int^t_0e^{-\delta(t-s)}e^{\gamma (t-s)}e^{\gamma s}
\|(R\bw)(\cdot, s)\|_{L_q(\Omega)}\,ds \\
& \leq 
M\int^t_0[e^{-(\delta-\gamma_0)(t-s)}]^{1/p'+1/p}e^{\gamma (t-s)}e^{\gamma s}
\|(R\bw)(\cdot, s)\|_{L_q(\Omega)}\,ds
\\
&\leq M\Bigl(\int^t_0 e^{-(\delta-\gamma_0)(t-s)}\,ds\Bigr)^{1/p'}
\Bigl(\int^t_0e^{-(\delta-\gamma_0)(t-s)}
(e^{\gamma s}\|(R\bw)(\cdot, s)\|_{L_q(\Omega))})^p\,ds
\Bigr)^{1/p},
\end{align*}
which, combined with \eqref{e:8},  yields that
\begin{equation}\label{e:9}
\|e^{\gamma t}\bv\|_{L_p((0, \infty), L_q(\Omega))} \leq C\CF_\gamma
\end{equation}
for any $\gamma \leq \gamma_0$. 

Since $\bv$ satisfies the shifted equations:
$$R(\pd_t\bv + \eta\bv) - \dv(B\nabla\bv) = 
-\eta R\bw + \eta R\bv
\quad\text{in $\Omega\times(0, \infty)$}, \quad
B(\nabla\bv\cdot\bn)|_\Gamma=0, \quad \bv|_{t=0}=0,
$$
we have, analogously to \eqref{e:8},
$$\|e^{\gamma t}\bv\|_{L_p((0, \infty), H^2_q(\Omega))}
+ \|e^{\gamma t}\pd_t\bv\|_{L_p((0, \infty), L_q(\Omega))}
\leq C\|e^{\gamma t}(\bw, \bv)\|_{L_p((0, \infty), L_q(\Omega))},
$$
which, combined with \eqref{e:9} and \eqref{e:8}, yields that
$$\|e^{\gamma t}(\bv + \bw)\|_{L_p((0, \infty), H^2_q(\Omega))}
+ \|e^{\gamma t}\pd_t(\bv + \bw)\|_{L_p((0, \infty), L_q(\Omega))}
\leq C_\gamma
$$
for any $\gamma \leq \gamma_0$.  Therefore,  $\bu = \bv + \bw$ is a 
required solution, which completes the proof of Theorem \ref{thm:1.1global}. 
\hfill
\begin{minipage}[t]{1cm}
\flushright $\Box$
\end{minipage}. 
\vskip3mm
{\bf Acknowledgement.} The authors would like to thank the anonymous referee for suggesting including the 
exponential decay result (Theorem 9).


\end{document}